\documentclass[10pt, reqno, a4paper]{amsart}
\usepackage{amssymb,latexsym,amsmath,amsfonts,amsthm}
\usepackage{latexsym}
\usepackage[numbers,sort&compress]{natbib}
\usepackage{color}
\usepackage[mathscr]{eucal}
\usepackage{comment}
\usepackage{dsfont}

\addtocontents{toc}{\protect\protect\setcounter{tocdepth}{1}}

\usepackage[linkcolor=red, citecolor=red]{hyperref}
\hypersetup{colorlinks=true, urlcolor=blue}
\usepackage[nameinlink]{cleveref}
\usepackage{amsmath,amssymb,amsthm,graphicx,amsxtra, setspace}
\usepackage[top=2.7cm,bottom=2.7cm,left=2.7cm,right=2.7cm]{geometry}
\usepackage{mathrsfs}
\usepackage{alltt}
\usepackage{relsize}
\usepackage{hyperref}
\usepackage{aliascnt}
\usepackage{tikz}
\usepackage{mathtools}
\usepackage{multicol}
\usepackage{upgreek}
\usepackage{graphicx,type1cm,eso-pic,color}
\allowdisplaybreaks

\numberwithin{equation}{section}

\theoremstyle{definition}
\newtheorem{definition}{Definition}[section]

\theoremstyle{remark}
\newtheorem{remark}[definition]{Remark}
\theoremstyle{plain}
\newtheorem{theorem}[definition]{Theorem}
\newtheorem{lemma}[definition]{Lemma}
\newtheorem{proposition}[definition]{Proposition}

\newcommand{\vk}{\varkappa}

\newcommand{\al}{\alpha}
\newcommand{\lm}{\lambda}
\newcommand{\Lm}{\Lambda}
\newcommand{\tht}{\theta}

\newcommand{\no}{\nonumber}

\newcommand{\gm}{\gamma}

\newcommand{\pat}{\partial_t}
\newcommand{\pan}{\partial_\mathbf{n}}
\newcommand{\td}{\text{div}}

\newcommand\ba[1]{\overline{#1}}

\newcommand\io{\int_\Omega}

\newcommand{\mz}{\boldsymbol{\zeta}}
\newcommand{\bw}{\boldsymbol{w}}
\newcommand{\bu}{\boldsymbol{u}}
\newcommand{\bv}{\boldsymbol{v}}
\newcommand{\bg}{\boldsymbol{g}}
\newcommand{\bz}{\boldsymbol{z}}
\newcommand{\be}{\boldsymbol{e}}
\newcommand{\by}{\boldsymbol{y}}

\newcommand{\tphi}{\partial_t\varphi}
\newcommand{\vphi}{\varphi}

\newcommand{\ms}{\mathcal{S}}

\newcommand{\J}{\mathcal{J}}

\newcommand\V{\mathcal{V}}
\newcommand\W{\mathcal{W}}
\newcommand{\U}{\mathcal{U}}

\newcommand{\MH}{\mathbf{H}}
\newcommand{\MV}{\mathbf{V}}
\newcommand{\ML}{\mathbb{L}}

\def\dx{\,\textnormal{d}x}

\newcommand{\RR}{\mathbb{R}^2}
\newcommand{\rea}{\mathbb{R}}

\begin{document}
	\title[Optimal Control ]{Optimal Control of a Higher-order Cahn-Hilliard equation coupled with Brinkman equation}
	\author [Manika Bag]{Manika Bag\textsuperscript{1*}}


\date{\today}
\thanks
{
\textsuperscript{1} School of Mathematics,
			Indian Institute of Science Education and Research, Thiruvananthapuram (IISER-TVM),
			Maruthamala PO, Vithura, Thiruvananthapuram, Kerala, 695551, INDIA.  
			\textit{e-mail:} \texttt{manikabag19@iisertvm.ac.in, manikbag058@gmail.com}\\		
\textsuperscript{*}  Corresponding author\\
{\textbf{Acknowledgments:}} The author would like to thank the  Indian Institute of Science Education and Research, Thiruvananthapuram, for providing financial support and stimulating  research environment.  
}

\begin{abstract}
 In this work, we investigate optimal control of a Brinkman equation couple with sixth-order Cahn-Hilliard equation. The Cahn-Hilliard equation is endowed with a source term accounting for mass exchange and the velocity equation contains a non divergence-free forcing term, which act as distributed control variable. We consider the aforementioned system with constant mobility, viscosity and nonlinearity of double-well shape is regular. The cost functional of the optimal control problem contains a nondifferentiable term like the $L^1$-norm with sparsity constant $\kappa$, which leads to sparsity of optimal controls. We study the first order necessary optimality condition for both the case $\kappa=0$ and $\kappa>0.$ When the cost functional is differentiable, first order necessary optimality conditions are characterized by Lagrange multiplier method and for nondifferentiable case we have used the idea of Casas and Tr\"oltzsch from the paper (Math. control Relat. Fields, 10(3):527-546, 2020).  
\end{abstract}

\maketitle


\section{Introduction}
In recent years, sixth-order Cahn--Hilliard equations have attracted considerable attention. 
Models of this type appear in several physical contexts, including phase separation with pronounced anisotropy effects, atomistic descriptions of crystal growth, the evolution of crystalline surfaces with small slopes exhibiting faceting phenomena, ternary systems such as oil--water--surfactant mixtures, and blends of polymeric materials.  
We also refer to \cite{WWL09, WW11,SP13, SPsima13,SW20}, for further analytical and numerical investigations of these higher-order models. When coupled with fluid dynamics, these models offer a robust framework for analyzing the interaction between interfacial evolution, transport mechanisms, and hydrodynamic effects. In this article we are interested in one such coupling where we consider the Brinkman equation couple with sixth-order Cahn--Hilliard equations.

\subsection{Model and problem description.} Let $\Omega\subset\rea^3$ is a bounded and connected open set with smooth boundary $\Gamma$ and unit outward normal $\mathbf{n}$. Further, let $T>0$ denote some final time and we set
$Q=\Omega\times(0, T), \quad\Sigma=\Gamma\times(0, T)$.In this article we study the following optimal control problem:
\begin{align*}
    \textbf{(CP)}_\eta \quad \min_{\bg\in\U_{ad}}\J(\bv,\vphi, \bg)
\end{align*} 
subject to
\begin{equation}\label{CHB}
    \left\{
    \begin{aligned}
       &- \eta\Delta\bv+\lm(\vphi)\bv+\nabla p=\mu\nabla\vphi+\bg, \quad \td ~ \bv=0, \quad \text{in } Q, \\
       &\tphi+\bv\cdot\nabla\vphi-\td(m(\vphi)\mu)=S(\vphi),  \quad \text{in } Q,\\
       &\mu= -\Delta w+f'(\vphi)w+\nu w,  \quad \text{in } Q,\\
       &w= -\Delta\vphi+f(\vphi),
    \end{aligned}
    \right.
\end{equation}
complemented with boundary and initial conditions 
\begin{align}
    &\bv=0, \text{ and } \pan\mu=\pan w=\pan\vphi= 0  \quad \text{ on } \Sigma,\label{bdry CHB}\\
    &\vphi(0)=\vphi_0,  \quad \text{in } \Omega.\label{in CHB}
    \end{align}
In the above $\J$ is a quadratic cost functional and $\U_{ad}$ is admissible control; precisely define as below
\begin{enumerate}
    \item[(1)] $\beta_i$ and $\kappa$ are nonnegative real numbers for $i=1,2,3, 4.$
    \item[(2)] $\bv_Q\in L^2(0, T; \MH)$, $\vphi_Q\in L^2(Q)$ and $\vphi_T\in H$. 
\end{enumerate}
Furthermore, we define the cost functional, the control space and admissible control space by setting
\begin{align}
  & \J(\bv,\vphi,\bg)=\frac{\beta_1}{2}\int_Q|\bv-\bv_Q|^2+\frac{\beta_2}{2}\int_Q|\vphi-\vphi_Q|^2 +\frac{\beta_3}{2}\int_\Omega|\vphi(T)-\vphi_T|^2+\frac{\beta_4}{2}\int_Q|\bg|^2+\kappa\int_Q|\bg|\label{cost functional}\\
  &\qquad\qquad\quad=J(\bv, \vphi, \bg)+\kappa j(\bg),\no\\
  & \U=L^2(0, T; \MH), \text{ and } \U_{ad}:=\{\bg\in\U:\|\bg\|_{\U}\leq M\}, \text{ $M$ being a positive constant}\label{contol sp}.\end{align} The corresponding functional spaces will be specified in section 2.

  The model \eqref{CHB} is studied in \cite{CGSS25}. The velocity field \(\bv\) is governed by a Brinkman-type equation see $\eqref{CHB}_1$.
The scalar field \(\vphi\), which acts as an order parameter, denotes the local fraction
of one component in a binary mixture. For analytical convenience, it is typically scaled
so that the pure phases correspond to \(\vphi = -1\) and \(\vphi = 1\), whereas values
in \((-1,1)\) represent the diffuse interfacial region. This interfacial layer is
concentrated in a tubular neighborhood of the physical interface and has thickness
proportional to a small parameter \(\varepsilon > 0\).

The dynamics of \(\vphi\) are described by a sixth-order Cahn--Hilliard-type equation,
modified by a nonconservative source term
\[
S(\vphi) = -\sigma \vphi + h(\vphi),
\]
where \(h\) is smooth and bounded and \(\sigma \in \mathbb{R}\). This term introduces
mass exchange and therefore destroys the standard mass-conservation property of \(\vphi\).

The quantities \(\mu\) and \(w\) appearing in $\eqref{CHB}_3-\eqref{CHB}_4$ correspond to the variational
derivatives of the total free energy \(\mathcal{E}\) and the Ginzburg--Landau energy \(\mathcal{G}\),
respectively:
\[
\mu := \frac{\delta \mathcal{E}}{\delta \vphi},
\qquad
w := \frac{\delta \mathcal{G}}{\delta \vphi},
\]
where
\[
\mathcal{E}(\vphi)
=
\frac{1}{2}
\int_{\Omega}
\left(-\varepsilon \Delta \vphi + \frac{1}{\varepsilon} f(\vphi)\right)^{2}
+ 
\nu \int_{\Omega}
\left(
\frac{\varepsilon}{2} |\nabla \vphi|^{2}
+
\frac{1}{\varepsilon} F(\vphi)
\right).
\tag{1.8}
\]

Here, \(F\) is a double-well potential defined on \(\mathbb{R}\), \(f = F'\), $\eta>0$ is viscosity coefficient and
\(\nu \in \mathbb{R}\) is a (possibly nonpositive) constant. A typical example is the
classical quartic potential
\[
F(s) = \frac{1}{4}(s^{2} - 1)^{2}, \qquad s \in \mathbb{R}.
\tag{1.9}
\]
The function \(\bg\) in $\eqref{CHB}_1$ acts as an optimal control. $\eta$ and $m$ denotes the viscosity and mobility of the fluids. The nonlinear term
\(\mu \nabla \vphi\) represents the Korteweg force, which accounts for capillarity effects
in the mixture. Due to the smoothness of \(f\), the boundary condition
\(\partial_{\mathbf{n}} w = 0\) in (1.5) is equivalent to
\(\partial_{\mathbf{n}} \Delta \vphi = 0\) on \(\Sigma\).
\subsection{State of the art and novelty.} The Brinkman equation coupled with the classical fourth-order Cahn–Hilliard dynamics has been extensively investigated in the literature. Without aiming to provide a complete overview, we briefly highlight several relevant contributions. In \cite{BCG15}, the authors established well-posedness, the existence of a global attractor, and the convergence of weak solutions to single equilibria for the Cahn–Hilliard–Brinkman (CHB) system. Optimal distributed control for the same model was later addressed in \cite{YL19}. Further analytical results, including the study of dynamic boundary conditions, can be found in \cite{LY21, CKSS24}. Within the context of tumor-growth modeling, related formulations of the CHB system have been analyzed in \cite{CGSS23, KS22, EG19, EK19, EGN21}. Nonlocal variants of the model were explored in \cite{DFG16, DM21}, where the former deals with well-posedness and the latter focuses on regularity properties and associated optimal control problems. The classical Cahn-Hilliard equation also studied with different other coupling such as Navier-Stokes and Brinkman equation \cite{arma, MTS, BDM24, EG19, EGN21, EK19}, Hele-Shaw equation \cite{WZ13, GGW18}, Boussinesq equation \cite{MGS17} etc.

We next turn to the sixth-order Cahn–Hilliard equation incorporating curvature effects. Well-posedness results for this higher-order system, under various choices of potentials and mobilities, are presented in \cite{M15, SW20, MZ20, Z21}, while optimal control aspects have recently been investigated in \cite{CGSS23}. Analytical studies of the sixth-order Cahn–Hilliard dynamics coupled with fluid flow models have also appeared, including the coupling with the Navier–Stokes equations \cite{SW25} and with the Brinkman equation \cite{CGSS25}.

None of the optimal control paper cited above is concerned with the aspect of \emph{sparsity}, i.e. the possibility that locally optimal control may vanish in subregion of positive measure of the space time cylinder $Q$ that are controlled by sparsity parameter $\kappa.$ The sparsity structure follows from the variational inequality arising in the first-order necessary optimality conditions, together with the particular form of the subdifferential $\partial j$. In this work, our interest lies in sparsity effects, and for simplicity, we restrict our attention to the case of full sparsity associated with the $L^1(Q)$- norm functional $j$ introduced in \eqref{cost functional}.

Sparsity in optimal control for partial differential equations has emerged as a highly active research area. The introduction of sparsity-promoting functionals can be traced back to inverse problems and image processing. A landmark contribution was made in the seminal study \cite{S09} on elliptic control problems, which ignited the investigation of sparsity in PDE-based optimal control. After that, numerous results on sparse optimal controls for PDEs have appeared. For instance, sparse control problems for the viscous Cahn–Hilliard equation, the Allen–Cahn equation, and certain tumor growth models are addressed in \cite{CST24, ST24, ST23, H23}. To the best of our knowledge, however, sparse optimal control for the Brinkman flow coupled with a curvature-driven sixth-order Cahn–Hilliard equation has not been investigated in the literature. The present work appears to be the first to consider this setting.

In this article, we study the associated optimal control problem and incorporate an $L^1(Q)$-norm of the control into the cost functional in order to capture sparsity effects. We begin by establishing the differentiability of the control-to-state operator and then derive the optimality system using a Lagrange multiplier framework. Employing standard arguments, we show the existence of an optimal control for problem \textbf{(CP)} and obtain the first-order necessary optimality conditions. Because of the intrinsic complexity of the sixth-order Cahn–Hilliard structure, second-order sufficient optimality conditions are not addressed here and will be treated in a separate work.
\subsection{Organization of the paper.}
In the upcoming section, we will provide some preliminary results, comprehensive list of specific assumption on system parameters and recall the well-posedness result. The differentiability of control-to-state operator is studied in section 3. The control problem for $\textbf{(CP)}_\eta$ is investigated in section 4 for both smooth and non smooth cost functional.

We close this section by introducing some convention concerning constants used in estimated within the paper. The constants $C$ and $c$ denotes any positive constants that depends only on given data occurring in the state system, cost functional, system parameters such as $\eta, \lm, \nu$, as well as constant $M$ that appear in the definition of $\U_{ad}.$ The actual value of such generic constants may changes from formula to formula and sometimes even within the formula. Moreover, $c_{a,b}$ indicates a positive constant that additionally depends on the quantity $a, b.$ 
\section{Notation, assumption and preliminary}
 For any Banach space $X$, we denote $\|\cdot\|_X$ and $X'$ as the corresponding norm and its dual space with few exceptional notation listed below. The classical Lebesgue and Sobolev space on $\Omega$ corresponding to each $1\leq p<\infty$ and $k\geq 0$ are denoted by $L^p(\Omega)$, $W^{k,p}(\Omega)$. Their associated norms are denoted as $\|\cdot\|_{L^p}$, $\|\cdot\|_{W^{k,p}}$, and $\|\cdot\|_\infty$ denots the $L^\infty-$norm. For $p=2$, we simply denote $\|\cdot\|$ as $L^2-$norm. Also we set 
$$H=L^2(\Omega), \, V=H^1(\Omega), \, W:=\{z\in H^2(\Omega):\pan z=0 \text{ on }\Gamma\},$$
\begin{align*}
    &\mathbb{C}_c^\infty(\Omega) := \big\{\bv \in C^\infty(\Omega:\mathbb{R}^3): \text{supp $\bv$ is compact}\big\}, \,  \mathbb{D}_\sigma(\Omega) := \big\{\bv \in \mathbb{C}_c^\infty(\Omega) : \text{div}~\bv = 0\big\},\\
    &\mathbb{L}^p_{\sigma}(\Omega) := \text{closure of $\mathbb{D}_\sigma(\Omega)$ in $\mathbb{L}^p(\Omega$}), \,
    \MV^s_{\mathrm{div}}(\Omega) := \text{closure of $\mathbb{D}_\sigma(\Omega)$ in $\mathbb{W}^{s,2}(\Omega)$} \quad \text{for $s > 0$}.
\end{align*}
We denote the Hilbert spaces $\mathbb{L}^2(\Omega;\rea^3)$ by $\MH$, $ \MV^1_{\mathrm{div}}(\Omega)$ by $\MV_{\td}.$ The inner product  and norm in $\MH$ is denoted by $(\cdot,\cdot)$ and $\|\cdot\|$, respectively. The duality between $\MV_{\td}$ and $\MV'_{\td}$ or $V$ and $V'$ by $\langle\cdot,\cdot\rangle.$ We also recall the standard notation for scaler product and the norm of matrices, namely
$$A:B=\sum_{i,j=1}^2a_{ij}b_{ij}, \, \text{ and } |A|^2=A:A \text{ for } A=(a_{ij}), B=(b_{ij})\in \RR.$$
We now list down some assumption on $\eta, \lm, m$, and $F$ as follows
\begin{itemize}
    \item[(A1)] $\eta$ is positive constant and $m(\vphi)\equiv 1$.
    \item[(A2)] $\lm:\mathbb{R}\to\mathbb{R}$ are locally Lipschitz continuous and positive, i.e. $ 0<\lm_\ast\leq\lm(s)\leq\lm_\ast,$  for every $s\in\mathbb{R}$ and some positive constants $\lm_\ast,\lm^\ast$.
    \item[(A3)] $\nu, \sigma\in\mathbb{R}$ and $S:\mathbb{R}\to\mathbb{R}$ is given by $S(\vphi)=-\sigma\vphi+h(\vphi)$, where $h:\mathbb{R}\to\mathbb{R}$ is bounded class $C^2$ and Lipschitz continuous.
    \item[(A4)]  $F:\mathbb{R}\to\mathbb{R}$ is of class $C^5$ and $f:=F'$ denotes its derivative;$$\lim_{|s|\to +\infty}\frac{f(s)}{s}=+\infty.$$
    \item[(A5)] $f'(s)\geq C_1$, $|F(s)|\leq C_2(|sf(s)|+1)$, and $|sf'(s)|\leq C_3(|f(s)|+1),$ for every $s\in \mathbb{R}$ and some positive constants $C_1, C_2, C_3.$ 
\end{itemize} 
We note that the last inequality in (A5) and continuity of $f'$ imply that
$$|f'(s)|\leq C_3'(|f(s)|+1), \quad \text{for all } s\in\rea,$$ for some constant $C_3'>0.$\\
We recall some useful inequalities that is used in the paper frequently. 

\noindent \textbf{Young's inequality.} Let $a,b$ be any non-negative real numbers. Then for any $\varepsilon>0,$ the following inequalities hold:
\begin{align} \label{eqPR-YoungIneq}
	ab \le \frac{\varepsilon a^2}{2} + \frac{b^2}{2\varepsilon} \text{ and } ab \le \frac{a^p}{p}+\frac{b^q}{q},
\end{align}
for any $p,q>1$ such that $\frac{1}{p}+\frac{1}{q}=1.$
\begin{proposition}[Generalized H\"older's inequality]{} \label{ppsPR-GenHoldIneq}
	Let $f\in L^p(\Omega),$ $g\in L^q(\Omega),$ and $h\in L^r(\Omega),$ where $1\le p, q ,r\leq \infty$ are such that $\frac{1}{p}+\frac{1}{q}+\frac{1}{r}=1.$ Then $fgh\in L^1(\Omega)$ and 
	\begin{align*}
		\|fgh\|_{L^1(\Omega)} \le \|f\|_{L^p(\Omega)} \|g\|_{L^q(\Omega)} \|h\|_{L^r(\Omega)}.
	\end{align*}
  \end{proposition}
 \begin{lemma}[{General Gronwall inequality \cite[p. 1647]{Canon99}}] \label{lem:Gronwall} 
	Let $f,g, h$ and $y$ be four locally integrable non-negative functions on $[t_0,\infty)$ such that
	\begin{align*}
		y(t)+\int_{t_0}^t f(s)ds \le C+ \int_{t_0}^t h(s)ds + \int_{t_0}^t g(s)y(s)ds\  \text{ for all }t\ge t_0,
	\end{align*}
	where $C\ge 0$ is any constant. Then 
	\begin{align*}
		y(t)+\int_{t_0}^t f(s)ds \le \left(C+\int_{t_0}^t h(s)ds\right) \exp\left( \int_{t_0}^t g(s)ds\right) \text{ for all }t\ge t_0.
	\end{align*}
\end{lemma}
We also account for the Sobolev and Poincaré inequalities, as well as for some inequalities associated with the elliptic regularity theory and the compact embeddings between Sobolev spaces (via Ehrling’s lemma). More precisely, the following estimates hold:
\begin{align}
&\|z\|_{q} \leq C_S \|z\|_{V}, &&\text{for every } z \in V \text{ and } q \in [1,6], \label{sobolev_ineq}\tag{2.44}\\[4pt]
&\|z\|_{\infty} \leq C_S \|z\|_{W}, &&\text{for every } z \in W, \tag{2.45}\\[4pt]
&\|z\|_{V} \leq C_P(\|\nabla z\| + |\overline{z}|), &&\text{for every } z \in V, \tag{2.46}\\[4pt]
&\|z\|_{W} \leq C_E (\|\Delta z\| + \|z\|), &&\text{for every } z \in W, \tag{2.47}\\[4pt]
&\|z\|_{H^4(\Omega)} \leq C_E (\|\Delta^2 z\| + \|z\|), &&\text{for every } z \in H^4(\Omega) \text{ with } z,\, \Delta z \in W, \tag{2.48}\\[4pt]
&\|z\|_{V} \leq \delta \|\Delta z\| + C_{\delta}\|z\|, &&\text{for every } z \in W \text{ and every } \delta>0, \tag{2.49}\\[4pt]
&\|z\|_{H^3(\Omega)} \leq \delta \|\Delta^2 z\| + C_{\delta}\|z\|_{V}, &&\text{for every } z \in H^4(\Omega) \text{ with } z,\, \Delta z \in W \text{ and every } \delta>0. \tag{2.50}\label{GN-H3}
\end{align}
Now we state some necessary results from \cite{CGSS25}. First we recall the definition of weak solution for the system \eqref{CHB}-\eqref{in CHB}.
\begin{definition}
    A quadruplet $(\bv, \vphi, \mu, w)$ with the regularity
    \begin{equation}\label{reg1}
        \left\{
    \begin{aligned}
        &\bv\in L^2(0, T; \MV_{\td}),\\
        &\vphi\in H^1(0, T; V')\cap L^\infty(0, T; W)\cap L^2(0, T; H^5(\Omega)),\\
        &\mu\in L^2(0, T; V),\\
        & w\in L^\infty(0, T; H)\cap L^2(0, T; H^3(\Omega)\cap W),
    \end{aligned} 
    \right.
     \end{equation}is said to be weak solution of \eqref{CHB} if that solves the variational inequality
    \begin{align}
&\int_\Omega\big(\eta D\bv:\nabla\mz +\lm(\vphi)\bv\cdot\mz\big)=\int_\Omega\big(\mu\nabla\vphi+\bg)\cdot\mz \text{ for every } \mz\in\MV_{\td} \text{ and a.e. in } (0, T),\label{v weak form}\\
&\langle\tphi, z\rangle+\int_\Omega\bv\cdot\nabla\vphi z+\int_\Omega \nabla\mu\cdot\nabla z = \int_\Omega S(\vphi)z \text{ for every } z\in V \text{ and a.e. in } (0, T),\\
& \al\int_\Omega(\nabla w\cdot\nabla z+f'(\vphi)wz)+\nu\int_\Omega wz =\int_\Omega\mu z \text{ for every } z\in V \text{ and a.e. in } (0, T),\\
& \int_\Omega\nabla\vphi\cdot\nabla z +\int_\Omega f(\vphi)z=\int_\Omega wz \text{ for every } z\in V \text{ and a.e. in }\label{w weak form} (0, T),
    \end{align} as well as the initial condition \begin{align}
        \vphi(0)=\vphi_0.
    \end{align}
\end{definition}  
Next we state the well-posedness results for the system that studied in \cite{CGSS25}.
\begin{theorem}\cite[Theorem 2.1, Theorem 2.3]{CGSS25}\label{CHB-wellposed}
\begin{itemize}
    \item[(i)]    Let $\bg\in L^2(0, T; \MH)$ and $\vphi_0\in W.$ 
Then, under the assumptions \((A1)–(A5)\)  there exists at least one quadruplet 
\((\bv, \vphi, \mu, w)\) that fulfills the regularity requirements \eqref{reg1}, solves Problem \eqref{v weak form}-\eqref{w weak form}, and satisfies the estimate
\begin{align}
    \|\bv\|_{L^2(0,T;\MV_{\td})} 
+ &\|\vphi\|_{H^1(0,T;V') \cap L^\infty(0,T;W) \cap L^2(0,T;H^5(\Omega))} 
+ \|\mu\|_{L^2(0,T;V)} 
\nonumber\\&\quad+ \|w\|_{L^\infty(0,T;H) \cap L^2(0,T;H^3(\Omega) \cap W)} 
\le K_1,\label{unif_est_weak sol}
\end{align}
with a constant \(K_1\) that depends only on the structure of the system, \(\Omega\), \(T\), and an upper bound for the norms of the data related to \(\bg\) and \(\vphi_0\).
\item[(ii)] 
Let  us suppose that \(\eta\) and \(m\) are positive constants. 
Assume further that condition \(\vphi\in W\) holds for the initial datum.  
For any \(\bg_i \in L^2(0,T;H)\), \(i=1,2\), let \((\bv_i, \vphi_i, \mu_i, w_i)\) denote the corresponding solutions.  
Then the following stability estimate is valid:
\begin{align}\label{cont-depend-est}
    \|\bv\|_{L^2(0,T;\MV_{\td})} &
+ \|\vphi\|_{C^0([0,T];V) \cap L^2(0,T;H^4(\Omega))} 
+ \|\mu\|_{L^2(0,T;H)} 
+ \|w\|_{L^2(0,T;W)} 
\no\\&\quad\le K_2 \|\bg\|_{L^2(0,T;H)},
\end{align}
where \((\bv, \vphi, \mu, w) = (\bv_1-\bv_2, \vphi_1-\vphi_2, \mu_1-\mu_2, w_1-w_2)\) and \(\bg = \bg_1 - \bg_2\).  
The constant \(K_2\) depends only on the structure of the system, the domain \(\Omega\), the final time \(T\), the initial datum \(\varphi_0\), and an upper bound for the \(L^2(0,T;H)\)-norms of \(\bg_1\) and \(\bg_2\).
\end{itemize}
\end{theorem}

\section{Differentiability of Control to State operator}
In this section we will prove Fr\'echet differentiability of control-to-state operator introduced in this section. In addition to the control space $\U$ and admissible control space $\U_{ad}$ define in \eqref{contol sp}, let us set the spaces 
$$\V=L^2(0, T; \MV_{\td})\times[C^0([0, T]; W)\cap L^2(0, T; H^5(\Omega))\cap H^1(0, T; V')],$$
$$\W=L^2(0, T; \MV_{\td})\times[C^0([0, T]; V')\cap L^2(0, T; H^4(\Omega))\cap H^1(0, T; W'),$$
and define the control-to-state operator as 
$$\ms:\U_{ad}\to\V, \text{ by }\ms(\bg)=(\bv, \vphi),$$
where $(\bv, \vphi, \mu, w)$ is the unique solution associated to the problem \eqref{CHB}-\eqref{in CHB}. It follows from \eqref{cont-depend-est} that the control-to-state operator is locally Lipschitz continuous from $\U$ to $\V.$

\subsection{Linearized System.} In order to study the Fr\'echet differentiability of $\ms$, we first study the following linearized system. Let $\bg\in\U_{ad}$ be a control with corresponding state $\ms(\bg)=(\bv,\vphi)\in\V$, where $\bv, \vphi$ is the first two component of $(\bv,\vphi,\mu, w)$. Then for every $\bu\in\U$ we consider the following system which is obtain by linearising \eqref{CHB} around $(\bv,\vphi,\mu, w)$:
\begin{equation}\label{lin-system}
    \left\{
    \begin{aligned}
      &-\eta\Delta\bw+\lm'(\vphi)\psi\bv+\lm(\vphi)\bw+\nabla q=\tht\nabla\vphi+\mu\nabla\psi+\bu \quad \text{in }Q\\
      &\td~\bw=0 \quad \text{in }Q\\
      &\pat\psi+\bw\cdot\nabla\vphi+\bv\cdot\nabla\psi-\Delta\tht=-\sigma\psi+h'(\vphi)\psi \quad \text{in }Q\\
      &\tht=-\Delta\xi+f''(\vphi)\psi w+f'(\vphi)\xi+\nu \xi \quad \text{in }Q\\
      &\xi=-\Delta\psi+f'(\vphi)\psi \quad \text{in }Q
    \end{aligned}
    \right.
\end{equation}
with boundary and initial data 
\begin{equation}\label{lin-in-bdry}
    \left\{
    \begin{aligned}
       &\bw=0,\quad \pan\psi=\pan\Delta\psi=\pan\tht=0 \quad \text{ on }\Sigma\\
       &\psi(0)=0 \quad\text{ in }\Omega.
    \end{aligned}
    \right.
\end{equation}
Let us prove the existence of weak solution of the system \eqref{lin-system}-\eqref{lin-in-bdry}. To derive the forthcoming results, we strongly rely on the estimates obtained in Theorem \ref{CHB-wellposed}. 
\begin{proposition}\label{exist_lin_sys}
 Let $T>0$, $\eta, m, \lm, \sigma, \nu, h, f$ satisfies the assumption (A1)-(A5), and $\vphi_0\in W.$ Let the control $\bg\in\U_{ad}$ with corresponding state $(\bv,\vphi,\mu, w)$ given by Theorem \ref{CHB-wellposed}. Then, for every $\bu\in \U,$ the linearized system \eqref{lin-system}-\eqref{lin-in-bdry} admits a unique solution $(\bw, \psi, \tht, \xi)$ on $[0, T]$ such that
 \begin{equation}
     \left\{
     \begin{aligned}
         &\bw\in L^2(0, T;\MV_\td)\\
         &\psi\in H^1(0, T; V')\cap L^\infty(0, T; W)\cap L^2(0, T; H^5(\Omega)),\\
        &\tht\in L^2(0, T; H),\\
        & \xi\in L^\infty(0, T; H)\cap L^2(0, T; H^3(\Omega)\cap W).
     \end{aligned}
     \right.
 \end{equation}
 Moreover, the following variational equality satisfied:
 \begin{align}
     &\eta\int_\Omega\nabla\bw :\nabla\bz dx+\int_\Omega\lm'(\vphi)\psi\bv\cdot\bz dx +\int_\Omega\lm(\vphi)\bw\cdot\bz \dx=\int_\Omega\tht\nabla\vphi\cdot\bz dx+\int_\Omega\mu\nabla\psi\cdot\bz dx+\int_\Omega\bu\cdot\bz dx,\label{lin-var-bw}\\
     &\int_\Omega\pat\psi\rho dx+\int_\Omega (\bw\cdot\nabla\vphi)\rho dx+\int_\Omega(\bv\cdot\nabla\psi)\cdot\rho dx -\int_\Omega\tht\cdot\Delta\rho dx=-\int_\Omega\sigma\psi\rho dx+\int_\Omega h'(\vphi)\psi\rho dx,\\
     &\int_\Omega\nabla\xi\cdot\nabla\rho dx +\int_\Omega f''(\vphi)\psi w\rho dx+\int_\Omega f'(\vphi)\xi\rho dx+\nu\int_\Omega\xi\rho dx=\int_\Omega\tht\rho dx,\\
     &\int_\Omega\nabla\psi\cdot\nabla\rho dx+\int_\Omega f'(\vphi)\psi\rho dx=\int_\Omega\xi\rho dx,\label{lin-var-xi}
 \end{align} for a.e.  in $(0, T)$ and for every $\bz\in\MV_\td$, $\rho\in W$ with $\psi(0)=0.$ Furthermore, the following estimate 
 \begin{align}\label{lin-energy}
     \|\bw\|_{L^2(0, T; \MV_{\td})}+\|\psi\|_{C^0([0, T]; V')\cap L^2(0, T; H^4)\cap H^1(0, T; W')}+\|\tht\|_{L^2(0, T; H)}+\|\xi\|_{L^2(0, T; W)}\leq K_3\|\bu\|_{L^2(0, T; \MH)}
 \end{align} holds true with a positive constant $K_3$ depending on $\Omega, T,$ the structure of the system and upper bound of norm of control $M.$
\end{proposition}
\begin{proof}
    We will prove the theorem using Faido-Galerkin approximation technique; the idea closely follows \cite[Theorem 2.1]{CGSS25}.

   Let \((\gamma_j, e_j)_{j\ge 1}\) denote the eigenpairs of the Neumann Laplacian, and
\((\boldsymbol{\gamma}_j, \mathbf{e}_j)_{j\ge 1}\) those of the Stokes operator. Then
\(\{e_j\}_{j\ge 1}\) and \(\{\mathbf{e}_j\}_{j\ge 1}\) form complete orthonormal bases of
\(H\) and \(\MH\), respectively. Set 
    \begin{align}\label{eigen-approx}
        V_n:=\emph{span}\{e_1, \cdots, e_n\}\quad \text{ and } \MV_n:=\emph{span}\{\be_1, \cdots, \be_n\}, \quad \text{ for } n=1,2,\cdots.
    \end{align}
    We note that each space \(V_n\) is contained in \(W\), and that \(\Delta z \in V_n\) for every \(z \in V_n\). In addition, since the constant functions are included in \(V_1\), they also belong to all \(V_n\). Furthermore, the union of the spaces \(V_n\) is dense in both \(V\) and \(H\). Analogously, the union of the spaces \(\MV_n\) is dense in \(\MV_{\td}\) as well as in \(\MH\).

\textbf{Galerkin Scheme. } We look for a quadruplet $(\bw_n, \psi_n, \tht_n, \xi_n)\in H^1(0,T; \MV_n)\times H^1(0, T, V_n)\times L^2(0, T; V_n)\times L^2(0, T, V_n)$ of the form 
\begin{align*}&\bw_n(t)=\sum_{j=1}^nw_{nj}(t)\be_j, \quad \psi_n(t)=\sum_{j=1}^n\psi_{nj}(t)e_j, \\ &\tht_n(t)=\sum_{j=1}^n\tht_{nj}e_j, \quad \xi_n(t)=\sum_{j=1}^n\xi_{nj}e_j\text{  for a.e } t\in(0,T)\end{align*} where the coefficient require to satisfy $w_{nj}, \psi_{nj}\in H^1(0, T), \, \tht_{nj}, \xi_{nj}\in L^2(0,T)$ and solves the variational equation
\begin{align}
     &\frac{1}{n}\int_\Omega\pat\bw_n\bz dx+\eta\int_\Omega\nabla\bw_n :\nabla\bz dx+\int_\Omega\lm'(\vphi)\psi_n\bv\cdot\bz dx +\int_\Omega\lm(\vphi)\bw_n\cdot\bz \dx=\int_\Omega\tht_n\nabla\vphi\cdot\bz dx\no\\&\quad+\int_\Omega\mu\nabla\psi_n\cdot\bz dx+\int_\Omega\bu\cdot\bz dx \text{ for every }\bz\in\MV_n \text{ and a.e. in } (0, T),\label{gs-w}\\
     &\int_\Omega\pat\psi_n\rho dx+\int_\Omega (\bw_n\cdot\nabla\vphi)\rho dx+\int_\Omega(\bv\cdot\nabla\psi_n)\cdot\rho dx +\int_\Omega\nabla\tht_n\cdot\nabla\rho dx=-\sigma\int_\Omega\psi_n\rho dx+\int_\Omega h'(\vphi)\psi_n\rho dx\no\\ &\text{ for all }\rho\in V_n  \text{ and a.e. in } (0, T),\label{gs-psi}\\
     &\int_\Omega\nabla\xi_n\cdot\nabla\rho dx +\int_\Omega f''(\vphi)\psi_n w\rho dx+\int_\Omega f'(\vphi)\xi_n\rho dx+\nu\int_\Omega\xi_n\rho dx=\int_\Omega\tht_n\rho dx,\label{gs-tht}\\
     &\int_\Omega\nabla\psi_n\cdot\nabla\rho dx+\int_\Omega f'(\vphi)\psi_n\rho dx=\int_\Omega\xi_n\rho dx \quad\text{ for every }\rho\in V_n \text{ and a.e. in } (0, T),\label{gs-xi}\\
     &\bw_n(0)=0, \quad \psi_n(0)=0.\label{gs-in}
 \end{align}
 Taking $\bz=\be_j$ in \eqref{gs-w} and $\rho=e_j$ in \eqref{gs-psi}-\eqref{gs-xi} for $j=1,2,\cdots n$, and using the orthogonality properties of the eigenfunctions, we obtain a system of ODE of the form
 \begin{align}
     &\frac{1}{n}w_{nj}'=\mathcal{K}_1\big((w_{nj})_{j=1}^n, (\psi_{nj})_{j=1}^n, (\tht_{nj})_{j=1}^n\big)\\
     &\psi_{nj}'=\mathcal{K}_2\big((w_{nj})_{j=1}^n, (\psi_{nj})_{j=1}^n, (\tht_{nj})_{j=1}^n\big)\\
     &\tht_{nj}=\mathcal{K}_3\big((\psi_{nj})_{j=1}^n, (\xi_{nj})_{j=1}^n\big),\\
     &\xi_{nj}=\mathcal{K}_4\big((\psi_{nj})_{j=1}^n\big),
 \end{align}
for all $j=1,2, \cdots n$ and a.e. with respect to time, the functions $\mathcal{K}_1$, $\mathcal{K}_2$, $\mathcal{K}_3$, $\mathcal{K}_4$ are naturally defined and locally Lipschitz continuous on $\rea^n\times\rea^n\times\rea^n$, due to linearity of the terms and (A2). Then proceeding similarly as in Step 1 of \cite[Theorem 2.1]{CGSS25} we obtain that \eqref{gs-tht}-\eqref{gs-in} has a unique maximal solution in $[0, T_n).$\\
We now begin the estimates needed to ensure that the solution to \eqref{gs-w}–\eqref{gs-in} is global, i.e., that \(T_n = T\). To simplify the notation, we drop the subscript \(n\) and write \(T\) in place of \(T_n\) from this point onward. Moreover, for any time-dependent test function \(z\), it is implicitly understood that each equation is considered at time \(t\) and tested with \(z(t)\) for a.e.\ \(t \in (0,T)\), although we omit the explicit notation for clarity.

\textbf{A priori estimates.}
Let $\rho=1$ in \eqref{gs-psi}, then we have 
\begin{align*}
    \frac{d}{dt}\io\psi+\sigma\io\psi=\io h'(\vphi)\psi
\end{align*} since $\psi(0)=0$ and using (A3) we obtain 
\begin{align}
   \left\| \io\psi\right\|_{L^\infty(0, T)}\leq c.\label{mean-psi}
\end{align}
Now let us first combine \eqref{gs-tht} and \eqref{gs-xi} and write
\begin{align}\label{combine-tht-xi}
    \io\nabla(-\Delta\psi+f'(\vphi)\psi)\cdot\nabla\rho+\io(f'(\vphi)+\nu)(-\Delta\psi+f'(\vphi)\psi)\rho+\io f''(\vphi)\psi w\rho =\io\tht\rho.
\end{align} 
Let $\bz=\bw$ in \eqref{gs-w}, 
 we find 
\begin{align}
&\io|\bw|^2+\io(\eta|\nabla\bw|^2+\lm(\vphi)|\bw|^2)=-\io\lm'(\vphi)\psi\bv\cdot\bw+\io\tht\nabla\vphi\cdot\bw+\io\mu\nabla\psi\cdot\bw+\io\bu\cdot\bw.
\end{align} 
Then taking $\al:=\min\{\eta, \lm_\ast\}>0$ yields
\begin{align}\label{test-w}
   \io|\bw|^2+\al\|\bw\|^2_{\MV_{\td}}\leq \io|\lm'(\vphi)\psi\bv\cdot\bw|+\io|\tht\nabla\vphi\cdot\bw|+\io|\mu\nabla\psi\cdot\bw|+\io|\bu\cdot\bw|.
\end{align} Next, we take $\rho=L\psi$ in \eqref{gs-psi} with $L$ being a constant whose value will be chosen later. Then we find
\begin{align}\label{test-psi}
    \frac{L}{2}\frac{d}{dt}\|\psi\|^2+L\io\nabla\tht\cdot\nabla\psi=-L\Big[\io\Big((\bw\cdot\nabla\vphi)\psi+(\bv\cdot\nabla\psi)\psi+\sigma\psi^2-h'(\vphi)\psi^2\Big)\Big].
\end{align} Note that second term of right hand side of \eqref{test-psi} is zero due to divergence free condition of $\bv$ and boundary condition \eqref{bdry CHB}. Next taking $\rho=-L\Delta\psi$ in \eqref{gs-psi} leads to
{\small\begin{align}\label{test-del-psi}
     \frac{L}{2}\frac{d}{dt}\|\nabla\psi\|^2+L\io\nabla\tht\cdot\nabla(-\Delta\psi)=-L\Big[\io\Big((\bw\cdot\nabla\vphi)(-\Delta\psi)+(\bv\cdot\nabla\psi)(-\Delta\psi)+\sigma\psi(-\Delta\psi)-h'(\vphi)\psi(-\Delta\psi)\Big)\Big].
\end{align}}
Now taking $\rho=-L\Delta\psi+L\Delta^2\psi$ and $-L\tht$ in \eqref{combine-tht-xi} we obtain
\begin{align}\label{test-lin-psi}
    L&\io|\nabla\Delta\psi|^2+L\io|\Delta^2\psi|^2=-L\io-\Delta (f'(\vphi)\psi)(-\Delta\psi+\Delta^2\psi)-L\io|f'(\vphi)|^2\psi(-\Delta\psi+\Delta^2\psi)\no\\&-L\io|f'(\vphi)|\Delta\psi|^2-L\io f'(\vphi)(-\Delta\psi)(\Delta^2\psi)-L\nu\io(-\Delta\psi)(-\Delta\psi+\Delta^2\psi)\no\\&-L\nu\io f'(\vphi)\psi(-\Delta\psi+\Delta^2\psi)-L\io f''(\vphi)\psi w(-\Delta\psi+\Delta^2\psi)+L\io\nabla\tht\cdot\nabla\psi+L\io\nabla\tht\cdot\nabla(-\Delta\psi)\Big)\Big],
\end{align}
as well as
\begin{align}\label{test-tht}
    L\io|\tht|^2&=L\io(\Delta^2\psi)\tht-L\io\Delta(f'(\vphi)\psi)\tht+L\io f'(\vphi)(-\Delta\psi)\tht+L\io |f'(\vphi)|^2\psi\tht\no\\&\quad+L\nu\io(-\Delta\psi)\tht+\nu L\io f'(\vphi)\psi\tht+L\io f''(\vphi)\psi w\theta.
\end{align}
Adding \eqref{test-w}-\eqref{test-tht} to each other and after a obvious cancellation then left hand side becomes
\begin{align}\label{lhs-add}
   \frac{1}{2n}\frac{d}{dt}\io|\bw|^2+\al\|\bw\|^2_{\MV_{\td}}+\frac{L}{2}\frac{d}{dt}\Big(\io|\psi|^2+|\nabla\psi|^2\Big)+L\io|\nabla\Delta\psi|^2+L\io|\Delta^2\psi|^2+L\io|\tht|^2,
\end{align}
and right hand side becomes
\begin{align}\label{rhs-add}
    &\io|\lm'(\vphi)\psi\bv\cdot\bw|+\io|\tht(\nabla\vphi\cdot\bw)| +\io|\mu(\nabla\psi\cdot\bw)|+\io|\bu\cdot\bw|-L\io(\bw\cdot\nabla\vphi)\psi-L\io(\bv\cdot\nabla\psi)\psi\no\\&-L\sigma\io\psi^2+L\io h'(\vphi)\psi^2-L\io(\bw\cdot\nabla\vphi)(-\Delta\psi)-L\io(\bv\cdot\nabla\psi)(-\Delta\psi)-L\sigma\io\psi(-\Delta\psi)\no\\&+L\io h'(\vphi)\psi(-\Delta\psi)-L\io-\Delta (f'(\vphi)\psi)(-\Delta\psi+\Delta^2\psi)-L\io|f'(\vphi)|^2\psi(-\Delta\psi+\Delta^2\psi)\no\\&-L\io f'(\vphi)|\Delta\psi|^2-L\io f'(\vphi)(-\Delta\psi)(\Delta^2\psi)-L\nu\io(-\Delta\psi)(-\Delta\psi+\Delta^2\psi)\no\\&-L\nu\io f'(\vphi)\psi(-\Delta\psi+\Delta^2\psi)-L\io f''(\vphi)\psi w(-\Delta\psi+\Delta^2\psi)+L\io(\Delta^2\psi)\tht-L\io\Delta(f'(\vphi)\psi)\tht\no\\&+L\io |f'(\vphi)|^2\psi\tht+L\io f'(\vphi)(-\Delta\psi)\tht+L\nu\io(-\Delta\psi)\tht+\nu L\io f'(\vphi)\psi\tht+L\io f''(\vphi)\psi w\tht.
\end{align}
We have to estimate the terms in \eqref{rhs-add}. In doing this we repeatedly make use of the inequalities \eqref{sobolev_ineq}-\eqref{GN-H3}, as well as of the H\"older and Young inequalities. For the first term we have 
\begin{align}
    \left| \int_{\Omega} \lambda'(\vphi) \psi \bv\cdot\bw \right|
\leq \|\lambda'(\vphi)\|_{\infty}  \|\psi\|_{L^4}\|\bv\| \|\bw\|_{L^4}
\leq \frac{\alpha}{8} \|\nabla \bw\|^2 + c_\al \|\psi\|_{V}^2 \|\bv\|^2.
\end{align}
Next we estimate the second term with explicit choice of parameter and get
\begin{align}
    \left| \int_{\Omega} \tht(\nabla\vphi\cdot\bw) \right|
&\leq \|\bw\|_{L^4} \|\nabla\vphi\|_{L^4} \|\tht\|
\leq \frac{\alpha}{8} \|\nabla\bw\|^2 + \frac{2 C_S^2}{\al} \|\varphi\|_{H^2}^2 \|\tht\|^2\no\\&\leq \frac{\alpha}{8} \|\nabla\bw\|^2 + \frac{2 C_S^2K_1^2}{\al}\|\tht\|^2,
\end{align}
where $C_S$ is the Sobolev embedding constant appear in \eqref{sobolev_ineq} and $K_1$ is the upper bound of $\|\vphi\|_{L^\infty(0, T; W)}$ which is bounded due to \eqref{unif_est_weak sol}. Next we have
\begin{align}
    &\left|\io\mu(\nabla\psi\cdot\bw)\right|\leq \|\bw\|_{L^4}\|\mu\|_{L^4}\|\nabla\psi\|\leq \frac{\alpha}{8} \|\nabla\bw\|^2+c_\al\|\mu\|_V^2\|\nabla\psi\|^2,\no\\
    &\left|\io\bu\cdot\bw\right|\leq\frac{\alpha}{8} \|\nabla\bw\|^2+c_\al\|\bu\|^2\no.
\end{align}
Now we estimate the RHS of \eqref{test-psi} that appear in \eqref{rhs-add}. We deduce that
\begin{align*}
   & -L\io(\bw\cdot\nabla\vphi)\psi\leq L\|\bw\|_{L^4}\|\nabla\varphi\|_{L^4}\|\psi\|\leq\frac{\al}{8}\|\nabla\bw\|^2+c_{\al,L}\|\vphi\|^2_{H^2}\|\psi\|^2,\\
   &-L\sigma\io\psi^2 \leq c_{L,\sigma}\|\psi\|^2,\\
   &-L\io h'(\vphi)\psi^2\leq c_L\|\psi\|^2,
\end{align*}
since $\vphi\in L^\infty(Q)$ and $h$ satisfy (A3), $h'(\vphi)$ is bounded in $\rea,$ which we have used in the last inequality. Now we address RHS of \eqref{test-del-psi} that appear in \eqref{rhs-add} taking into account integration by parts, Poincar\'e, H\"older, Young inequality and the estimate \eqref{sobolev_ineq}, \eqref{GN-H3}. We get that
\begin{align*}
   & -L\io(\bw\cdot\nabla\vphi)(-\Delta\psi)\leq L\|\bw\|_{L^4}\|\nabla\vphi\|_{L^4}\|\Delta\psi\|\leq\frac{\al}{8}\|\nabla\bw\|^2+\frac{2C_S^2}{\al}\|\vphi\|^2_{H^2}\|\Delta\psi\|^2\\&\qquad\qquad\leq\frac{\al}{8}\|\nabla\bw\|^2+\delta\|\Delta^2\psi\|^2+\frac{C^4_S\|\vphi\|^4_{H^2}}{2\al\delta}\|\psi\|^2_V\\
   &\qquad\qquad=\frac{\al}{8}\|\nabla\bw\|^2+\delta\|\Delta^2\psi\|^2+c_{\al,\delta,L}\|\vphi\|^4_{H^2}\|\psi\|^2_V,\\
&-L\io(\bv\cdot\nabla\psi)(-\Delta\psi)\leq L\|\bv\|_{L^4}\|\nabla\psi\|\|\Delta\psi\|_{L^4}\leq\delta\|\psi\|^2_{H^3}+c_{\delta, L}\|\nabla\bv\|^2\|\nabla\psi\|^2\\&\qquad\qquad\leq\delta\|\Delta^2\psi\|^2+c_\delta\|\psi\|^2_V+c_{\delta,L}\|\nabla\bv\|^2\|\nabla\psi\|^2,\\
&-L\sigma\io\psi(-\Delta\psi)\leq c_{L,\sigma}\|\nabla\psi\|^2,\\
&-L\io h'(\vphi)\psi(-\Delta\psi)\leq cL\|\psi\|\|\Delta\psi\|\leq\delta\|\Delta^2\psi\|^2+c_{\delta,L}\|\psi\|^2_V.
\end{align*}
We now estimate the terms on the right-hand side of \eqref{test-lin-psi} and \eqref{test-tht}. Although the first term in \eqref{test-tht} is straightforward to handle, we present its estimate explicitly, since the constant $L$ will later be chosen based on this expression. So we get
\begin{align*}
    L\io(\Delta^2\psi)\tht\leq\frac{L}{2}\|\Delta^2\psi\|^2+\frac{L}{2}\|\tht\|^2.
\end{align*}
Now we consider the first term of \eqref{test-lin-psi} and second term of \eqref{test-tht} involving a generic function  $z\in L^2(0, T, H)$ as they can be estimated similarly. So, we have
\begin{align*}
    &-L\io-\Delta(f'(\vphi)\psi)z= L\io\big(f'''(\vphi)|\nabla\vphi|^2\psi+f''(\vphi)\Delta\vphi\psi+2f''(\vphi)\nabla\vphi\cdot\nabla\psi+f'(\vphi)\Delta\psi\big)z\\&\leq L(\|f'''(\vphi)\|_\infty\|\nabla\vphi\|^2_{L^6}\|\psi\|^2_V+\|f''(\vphi)\|_\infty\|\Delta\vphi\|_{L^4}\|\psi\|_{L^4}+\|f''(\vphi)\|_\infty\|\nabla\vphi\|_{L^4}\|\nabla\psi\|_{L^4}\no\\&\qquad+\|f'(\vphi)\|_\infty\|\Delta\psi\|)\|z\|\\
    &\leq \delta\|z\|^2+C_Sc_{\delta,L}\|\vphi\|^2_{H^3}\|\psi\|^2_V+C_Sc_{\delta,L}\|\vphi\|^2_{H^2}\|\psi\|^2_{H^2}+c_{\delta,L}\|\Delta\psi\|^2\\
    &\leq \delta\|z\|^2+\delta\|\Delta^2\psi\|^2+c_{\delta,L}(\|\vphi\|^2_{H^3}+\|\vphi\|^4_{H^2})\|\psi\|^2_V.
\end{align*}
Now as explained in \cite{CGSS25}, by choosing $z=-\Delta\psi+\Delta^2\psi-\tht$ we derive the above estimate as 
\begin{align}
    &-L\io-\Delta(f'(\vphi)\psi)(-\Delta\psi+\Delta^2\psi-\tht)\no\\&\qquad\leq \delta\|^2-\Delta\psi+\Delta^2\psi-\tht\|+\delta\|\Delta^2\psi\|^2+c_{\delta,F,L}(\|\vphi\|^2_{H^3}+\|\vphi\|^4_{H^2})\|\psi\|^2_V,
\end{align} where again the first term can be estimated as
\begin{align}\label{3term expan}
    \delta\|-\Delta\psi+\Delta^2\psi-\tht\|\leq 4\delta\|\Delta^2\psi\|^2+3\delta\|\tht\|^2+c_\delta\|\psi\|^2_V.\end{align} The rest of the elements on the RHS of \eqref{test-lin-psi} and \eqref{test-tht} can be estimated similarly using (A5), \eqref{sobolev_ineq}, \eqref{GN-H3} and previously mentioned inequalities. We estimate them briefly as follows
    \begin{align}
        & -L\io|f'(\vphi)|^2\psi z\leq \delta\|z\|^2+c_{\delta, L}\|\psi\|^2,\\
        & -L\io f'(\vphi)(-\Delta\psi)z\leq \delta\|z\|^2+c_{\delta, F, L}\|\Delta\psi\|^2\no\\&\qquad\qquad\qquad\qquad\quad \leq \delta\|z\|^2+\delta\|\Delta^2\psi\|^2+\frac{C_{\delta, L}}{4}\|\psi\|^2_V,\\
        &-L\nu\io(-\Delta\psi)z \leq \delta\|z\|^2+\delta\|\Delta^2\psi\|^2+c_{\delta, \nu, L}\|\psi\|^2_V,\\
        &-L\nu\io f'(\vphi)\psi z\leq \delta\|z\|^2+c_{\delta, L}\|\psi\|^2,\\
        &-L\io f''(\vphi)\psi wz \leq \delta\|z\|^2+c_{\delta, L}\|w\|^2_V\|\psi\|^2_V.
    \end{align}
  We have estimated all the terms in \eqref{rhs-add}, putting the estimate \eqref{lhs-add} and \eqref{rhs-add} together yields
  \begin{align}
      &\frac{1}{2n}\frac{d}{dt}\io|\bw|^2+\frac{\al}{8}\|\bw\|^2_{\MV_{\td}}+\frac{L}{2}\frac{d}{dt}\io(|\psi|^2+|\nabla\psi|^2)+L\io|\nabla\Delta\psi|^2+\left(\frac{L}{2}-30\delta\right)\io|\Delta^2\psi|^2\no\\&+\left(\frac{L}{2}-\frac{2C_S^2K_1^2}{\al}-18\delta\right)\io|\tht|^2\leq c\|\bu\|^2+\Psi\|\psi\|^2_V,
  \end{align}
  where we denote
  \begin{align*}
      \Psi(t)=c(\|\bv(t)\|^2+\|\mu(t)\|^2_V+\|\vphi(t)\|^2_{H^2} +\|\vphi(t)\|^2_{H^3}+\|\vphi(t)\|^4_{H^2}+\|\bv(t)\|^2_{\MV_{\td}}+\|w(t)\|^2_V+1).
  \end{align*}
  Since $(\bv, \vphi, \mu, w)$ satisfy \eqref{unif_est_weak sol}, $\Psi$ belongs to $L^1(0, T)$. Now we are in position to apply Gronwall inequality; before that let us first choose $\delta=\frac{L}{120}$ and $L>\frac{40C_S^2K_1^2}{7\al}$, we conclude in particular that 
  \begin{align}
      &\frac{1}{n^{1/2}}\|\bw\|_{L^\infty(0, T; \MH)}+\frac{\al^{1/2}}{2}\|\bw\|_{L^2(0, T; \MV_{\td})}+L^{1/2}\|\psi\|_{L^\infty(0, T; V)}+(2L)^{1/2}\|\nabla\Delta\psi\|_{L^2(0, T; H)}\no\\&+L^{1/2}\|\Delta^2\psi\|_{L^2(0, T; H)}+\tilde{\al}\|\tht\|_{L^2(0, T; H)}<C, \quad\text{ with $\tilde{\al}>0$ is some constant}.
  \end{align} Now using elliptic regularity estimate we obtain 
   \begin{align}\label{uni-est-lin-var}
      &\frac{1}{n^{1/2}}\|\bw\|_{L^\infty(0, T; \MH)}+\frac{\al^{1/2}}{2}\|\bw\|_{L^2(0, T; \MV_{\td})}+L^{1/2}\|\psi\|_{L^\infty(0, T; V)}+L^{1/2}\|\psi\|_{L^2(0, T; H^4)}\no\\&+\tilde{\al}\|\tht\|_{L^2(0, T; H)}<C,\end{align}
      where the constant $C$ may depends on $ K_1$ and upper bound of the norm $\|\bu\|_{L^2(0, T;\MH)}$ but independent of $n.$
      
\textbf{Time-derivative estimate. } For $\rho\in L^2(0, T; W)$, we first estimate the following using \eqref{gs-psi} and for a.e. $t\in (0, T)$
\begin{align*}
    \langle\pat\psi, \rho\rangle_{W', W}\leq \|\bw\|_{L^6}\|\nabla\vphi\|\|\rho\|_{L^3}+\|\bv\|_{L^6}\|\nabla\psi\|\|\rho\|_{L^3} + \|\tht\|\|\Delta\rho\|+|\sigma|\|\psi\|\|\rho\|+c\|\psi\|\|\rho\|.
\end{align*} 
Then integrating the above estimate between $0$ to $T$ we obtain
\begin{align*}
    \int_0^T\langle\pat\psi, \rho\rangle_{W', W}\leq & c\Big(\|\bw\|_{L^2(0, T; \MV_{\td})}\|\vphi\|_{L^\infty(0, T; V)}+\|\bv\|_{L^2(0, T; \MV_{\td})}\|\psi\|_{L^\infty(0, T; V)}+\|\tht\|_{L^2(0, T; H)}\no\\&+\|\psi\|_{L^2(0, T; H)}\Big)\|\rho\|_{L^2(0, T; W)}.
\end{align*}
Since $\rho\in L^2(0, T; W)$ is arbitrary, by virtue of above estimate it is standard argument to pass to supremum with respect to $\rho$ and using \eqref{uni-est-lin-var} to get
\begin{align}\label{time unif lin var}
    \|\pat\psi\|_{L^2(0, T; W')}\leq C,
\end{align}
where $C$ is independent of $n.$

\textbf{Estimate on $\xi$.} To get an uniform bound on $\xi,$ we consider the equality \eqref{combine-tht-xi} 
\begin{align}\label{var-xi}
    \io\nabla(-\Delta\psi+f'(\vphi)\psi)\cdot\nabla\rho+\io(f'(\vphi)+\nu)(-\Delta\psi+f'(\vphi)\psi)\rho+\io f''(\vphi)\psi w\rho =\io\tht\rho,
\end{align} for $\rho\in W$ and a.e. $t\in (0, T)$. Considering $\xi=-\Delta\psi+f'(\vphi)\psi$, we infer from \eqref{uni-est-lin-var}
$$\|\xi\|_{L^2(0, T; H)}\leq C.$$
Also we note that $\xi$ solves the variational inequality \eqref{var-xi}, which can be re-written as
\begin{align}\label{elliptic-var-xi}
    \io\xi(-\Delta\rho+\rho) =\io\tilde{h}\rho, \quad\text{ for } \rho\in W, \text{ a.e. }t\in(0, T),
\end{align}
where $\tilde{h}=\tht-f''(\vphi)w\psi+(1-f'(\vphi)-\nu)\xi$ which is belongs to $L^2(0, T; H)$, thanks to \eqref{unif_est_weak sol} and \eqref{uni-est-lin-var}. Then from elliptic regularity results in \eqref{elliptic-var-xi} we infer that $\xi\in L^2(0, T; W)$ solves $-\Delta\xi+\xi=\tilde{h},$ a.e. in $Q$ and 
\begin{align}\label{unif-lin-xi}
    \|\xi\|_{L^2(0, T; W)}\leq c\|\tilde{h}\|_{L^2(0, T; H)}\leq C,\end{align} again $C$ is independent of $n.$  We note from from the estimates presented in \eqref{uni-est-lin-var}, \eqref{time unif lin var} and \eqref{unif-lin-xi}, that solution $(\bw_n, \psi, \tht_n, \xi_n)$ of the approximated problem \eqref{gs-w}-\eqref{gs-tht} is global i,e, $T_n=T.$

\textbf{Extraction of subsequences and limit passing. } From the above mentioned uniform estimates along with Banach-Alaoglu theorem, we can extract subsequences $(\bw_{n_k}, \psi_{n_k}, \tht_{n_k}, \xi_{n_k})_{k\in\mathbb{N}}$ (we will continue to denote as $(\bw_n, \psi, \tht_n, \xi_n)$) such that 
\begin{equation}
    \left\{
    \begin{aligned}
        &\bw_n\to\bw \quad\text{ weakly in }L^2(0, T; \MV_{\td}),\\
        &\psi_n \to \psi\quad\text{ weakly star in }L^\infty(0, T; V),\\
        &\psi_n\to \psi \quad\text{ weakly in } L^2(0, T; H^4),\\
        &\pat\psi_n\to \pat\psi\quad\text{ weakly in } L^2(0, T; W'),\\
        &\tht_n \to \tht \quad\text{ weakly in } L^2(0, T; H),\\
        &\xi_n \to \xi\quad\text{ weakly in } L^2(0, T; W).
    \end{aligned}
    \right.
\end{equation}
With this convergence results in hand, we are now left to show that the  quadruplet $(\bw, \psi, \tht, \xi)$ solves \eqref{lin-var-bw}- \eqref{lin-var-xi} and satisfy the estimate \eqref{lin-energy}. The last part is a easy consequence of estimates \eqref{uni-est-lin-var}, \eqref{time unif lin var}, \eqref{unif-lin-xi} and lower semicontinuity of norms. Now to show that $(\bw, \psi, \tht, \xi)$ is a solution of \eqref{lin-var-bw}-\eqref{lin-var-xi}, follows directly from conclusion part of \cite[Theorem 2.1]{CGSS25} (present case is more easier as the system is linear and some system parameter like $m$, $\eta$ are constants).

\textbf{Uniqueness. }Let us denote $\delta\bw=\delta\bw_1-\delta\bw_2,$ $\delta q=q_1-q_2$, $\delta\psi=\psi_1-\psi_2,$ $\delta\tht=\tht_1-\tht_2,$, $\delta\xi=\xi_1-\xi_2$, and $\delta\bu-\bu_1-\bu_2$ where $(\bw_i, q_i, \psi_i, \tht_i, \xi_i)$ for $i=1, 2$, are two solution of linearized system \eqref{lin-system}-\eqref{lin-in-bdry} corresponding to the control $\bu_i$. As the system is linear, $(\delta\bw, \delta\psi, \delta\tht, \delta\xi)$ again satisfy the same system \eqref{lin-system}-\eqref{lin-in-bdry}. Then by virtue of \eqref{lin-energy}, $(\delta\bw, \delta\psi, \delta\tht, \delta\xi)$ satisfy the estimate 
 \begin{align*}
     &\|\delta\bw\|_{L^2(0, T; \MV_{\td})}+\|\delta\psi\|_{C^0([0, T]; V')\cap L^2(0, T; H^4)\cap H^1(0, T; W')}+\|\delta\tht\|_{L^2(0, T; H)}+\|\delta\xi\|_{L^2(0, T; W)}\\&\qquad\leq K_3\|\delta\bu\|_{L^2(0, T; \MH)},
 \end{align*} where $K_3$ is the same constant appear in \eqref{lin-energy}. Therefore, if $\bu_1=\bu_2$ then solution of the system \eqref{lin-system}-\eqref{lin-in-bdry} is unique. 
\end{proof} 
Proposition \ref{exist_lin_sys} paves the way to proving Fr\'echet differentiability of control-to-state operator $\ms.$
\subsection{Differentiability of $\ms$}
Consider the control to state operator $ \mathcal{S}$ defined in section 3. Since $\V \subseteq \W$,  we can consider $\ms \text{ from } \mathcal{U} \text{ to the weaker space } \mathcal{W}$. So we consider the control-to-state map $\ms:\U\to\W.$
\begin{definition}\label{def:fd_S}
    We say $\ms: \mathcal{U} \rightarrow \W$ is Fr\'echet differentiable in $\mathcal{U}$ if for any $\bg \in \mathcal{U}$, there exist a linear operator $\mathcal{S}'(\bg): \mathcal{U}\rightarrow \mathcal{W}$ such that
\begin{align}\label{diff criteria}
    \lim_{\|\boldsymbol{\bu}\|_{\mathcal{U}}\rightarrow 0}\frac{\|\mathcal{S}(\bg+ \boldsymbol{u})-\mathcal{S}(\bg)-\ms'(\bg)(\bu)\|_{\W}}{\|\bu\|_{\mathcal{U}}} = 0,
\end{align}
for any arbitrary small perturbation $\bu \in \mathcal{U}_{ad}$.
\end{definition}
\begin{theorem}\label{thm-f-diff}
  Let the assumption (A1)-(A5) satisfied. The control to state operator, $\ms$ is Fr\'echet differentiable in $\mathcal{U}_{ad}$ as a mapping from $\U$ to $\V$, and the Fr\'echet derivative $\ms'(\bg)\in\mathcal{L}(\U, \V)$ is the linear operator that map any $\bu\in \U$ into the component $(\bw, \psi)$ of the solution of linearized system \eqref{lin-system} associated to $\bg$ and variation $\bu$, that is 
  $$\ms'(\bg)(\bu) = (\bw, \psi),$$ 
  where $(\bw, \psi)$ is the first two components of the quadruplet $(\bw, \psi, \tht, \xi)$, the unique solution of the linearized system \eqref{lin-system}.
\end{theorem}

\begin{proof}
 To prove the theorem, it is enough to show that \eqref{diff criteria} holds for every $\bu,\, \bg$ mentioned above. Without loss of generality, we can assume $\|\bu\|_{\U} < M - \|\bg\|_{\U}$, so that $\bg + \bu \in \U_{ad}.$  
 Also, we take $\ms(\bg+\bu) = (\bv^u, \vphi^u, \mu^u, w^u)$ and $\ms(\bg) = (\bv, \vphi, \mu, w)$ to be the solutions of the system \eqref{CHB} corresponding to the controls $\bg+\bu$ and $\bg$, respectively. For convenience, we set 
 $$\ba{\bv} = \bv^u - \bv,\quad \ba{p} = p^u - p,\quad \ba{\vphi} = \vphi^u - \vphi,\quad \ba{\mu} = \mu^u - \mu,\quad \ba{w} = w^u - w,$$
 then we have 
 \begin{equation}\label{1st diff CHB}
     \left\{
     \begin{aligned}
      & -\eta\Delta\ba{\bv}+\lm(\vphi)\ba{\bv}+\nabla\ba{p}=\ba{\mu}\nabla\vphi^u+\mu\nabla\ba{\vphi}-(\lm(\vphi^u)-\lm(\vphi))\bv^u+\bu \quad\text{ in }Q,\\
      &\pat\vphi+\ba{\bv}\cdot\nabla\vphi^u+\bv\cdot\nabla\ba{\vphi}-\Delta\ba{\mu}=-\sigma\ba{\vphi}+h(\vphi^u)-h(\vphi)\quad\text{ in }Q,\\
      &\ba{\mu}=\Delta\ba{w}+f'(\vphi)\ba{w}+(f'(\vphi^u)-f'(\vphi))w^u+\nu\ba{w}\quad\text{ in }Q,\\
      &\ba{w}=-\Delta\ba{\vphi}+f(\vphi^u)-f(\vphi),
     \end{aligned}
     \right.
 \end{equation}
 complemented with boundary and initial conditions 
\begin{align}
    &\ba{\bv}=0, \text{ and } \pan{\ba\mu}=\pan \ba{w}=\pan\ba{\vphi}= 0  \quad \text{ on } \Sigma,\label{bdry 1diffCHB}\\
    &\ba{\vphi}(0)=\vphi_0,  \quad \text{in } \Omega.\label{in 1diffCHB}
    \end{align}
  We note that the estimate \eqref{cont-depend-est} given by Theorem \ref{CHB-wellposed} holds for \eqref{1st diff CHB}, i.e.
  \begin{align}\label{cont-depend-est 3.34}
    \|\ba{\bv}\|_{L^2(0,T;\MV_{\td})} &
+ \|\ba{\vphi}\|_{C^0([0,T];V) \cap L^2(0,T;H^4(\Omega))} 
+ \|\ba{\mu}\|_{L^2(0,T;H)} 
+ \|\ba{w}\|_{L^2(0,T;W)} 
\no\\&\quad\le K_2 \|\bu\|_{L^2(0,T;\MH)},
\end{align}
The constant \(K_2\) depends only on the structure of the system, the domain \(\Omega\), the final time \(T\), and an upper bound of the norms $\|(\bv^u, \vphi^u, \mu^u, w^u)\|_{\V}$, $\|(\bv, \vphi, \mu, w)\|_{\V}$, $\|\bg\|_{\U}.$
Now, we set $$\bz=\ba{\bv}-\bw, \, r=\ba{p}-p, \, \zeta=\ba{\vphi}-\psi, \, \tau=\ba{\mu}-\tht, \, \gm=\ba{w}-\xi.$$ 
We are going to prove an inequality that imply \eqref{diff criteria}; namely we want to show the following estimate is satisfied: 
\begin{align}\label{Diff_est}
    \|\bz\|_{L^2(0, T, \MV_{\td})}+\|\zeta\|_{C^0([0,T];V') \cap L^2(0,T;H^4(\Omega))}+\|\tau\|_{L^2(0,T;H)}+\|\gm\|_{L^2(0,T;W)}\le c\|\bu\|^2_{L^2(0, T; \MH)}.
\end{align}
To this end, we note that $(\bz, \zeta, \tau, \gm)$ satisfy the following problem:
\begin{align}
    &\eta\io\nabla\bz:\nabla\by \, dx+\io\lm(\vphi)\bz\cdot\by \, dx=\io\ba{\mu}(\nabla\ba{\vphi}\cdot\by) \, dx, +\io\tau(\nabla\vphi\cdot\by) \, dx+\io\mu(\nabla\zeta\cdot\by) \dx+\io\Lm_1\cdot\by \, dx,\label{z-equ}\\
    &\langle\pat\zeta, y\rangle+\io\nabla\tau\cdot\nabla y \, dx=-\io(\bz\cdot\nabla\varphi)y \, dx-\io(\bv\cdot\nabla\zeta)y \, dx-\sigma\io\zeta y\, dx-\io(\ba{\bv}\cdot\nabla\ba{\vphi})y\, dx+\io\Lm_2 y\, dx,\label{zeta_equ}\\
    &\io\tau y dx= \io\nabla\gm\cdot\nabla y dx+\io(f'(\vphi)+\nu)\gm y\, dx+\io\Lm_3 y \, dx,\label{1st tau equ}\\
    &\io\gm y\, dx=\io\nabla\zeta\cdot\nabla y \, dx+\io\Lm_4y \, dx,\label{gm equ}
    \\
    & \quad\zeta(0)=0,\label{zeta_in}
    \end{align}
      for every $\by\in\MV_{\td}$, $y\in V,$  and a.e. in  $(0, T)$, where we have set 
      \begin{align*}
          &\Lm_1=(\lm(\vphi^u)-\lm(\vphi))\ba{\bv}+(\lm(\vphi^u)-\lm(\vphi)-\lm'(\vphi)\psi)\bv,\\
          &\Lm_2=h(\vphi^u)-h(\vphi)-h'(\vphi)\psi,\\
          &\Lm_3=(f'(\vphi^u)-f'(\vphi))\ba{w}+(f'(\vphi^u)-f'(\vphi)-f''(\vphi)\psi)w,\\
          &\Lm_4=f(\vphi^u)-f(\vphi)-f'(\vphi)\psi.
      \end{align*} Now combine \eqref{1st tau equ} and \eqref{gm equ} together to obtain \begin{align}
          &\io\tau y_1 \, dx=\io(-\Delta\zeta+\Lm_4)(-\Delta y_1) \, dx+\io(f'(\vphi)+\nu)(-\Delta\zeta+\Lm_4)y_1 \, dx+\io\Lm_3 y_1 \,dx,\label{tau_equ}
      \end{align}
     for every $y_1\in W$. Let us recall the Taylor formula for a $C^2$ function $g:\rea\to\rea$ with integral remainder as 
      \begin{align}\label{tylor}
          g(\vphi^u)-g(\vphi)-g'(\vphi)\psi=g'(\vphi)\zeta+R_g, \text{ with } R_g=\Big[\int_0^1(1-s)g''(s\vphi^u+(1-s)\vphi) \, ds\Big](\vphi^u-\vphi)^2.
      \end{align}
      Accordingly, we can rewrite $\Lm_i, i=1, \cdots 4$ using integral remainder $R_g$ for $g\in\{\lm, h, f, f'\}.$ Also note that, due to uniform boundedness of $\vphi^u$, $\vphi$ in $L^\infty(Q)$ and the integration variable s varies from $0$ to $1$, we have 
      \begin{align}\label{est-R}
          \|R_g\|_\infty \leq c\|\vphi^u-\vphi\|^2_\infty .
          \end{align}
          Now we need to find a solution of the variational problem \eqref{z-equ}--\eqref{zeta_in} that satisfies the estimate \eqref{Diff_est}. This can be achieved through a discretization procedure such as the Faedo--Galerkin scheme, employing the same set of eigenfunctions as those used in Proposition \ref{exist_lin_sys}. The analogous equations \eqref{z-equ}--\eqref{zeta_in}, satisfied for every $\by \in \MV_n$, and $y, y_1 \in V_n$, have their unknowns in \eqref{z-equ} taking values in $\MV_n$, while \eqref{zeta_equ}--\eqref{tau_equ} are $V_n$-valued. Thus, we obtain a system of ordinary differential equations, which can be solved by the Carathéodory existence theorem. Subsequently, one derives the required a priori estimates and passes to the limit as $n \to \infty$. 

For the sake of brevity, we argue only formally, noting that the test functions used here are admissible at the discrete level—particularly when considering the variational form of the problem, since $\Delta^k v \in V_n \subset W$ for every $k, n \in \mathbb{N}$ and $v \in V_n.$
      We test \eqref{z-equ}, \eqref{zeta_equ}, \eqref{tau_equ} by $\bz,$ $L(\zeta-\Delta\zeta)$, and  $L(-\Delta\zeta+\Delta^2\zeta-\tau)$, respectively, and following identities holds for a.e. in $(0, T)$:
      \begin{align}
          &\eta\|\nabla\bz\|^2+\io\lm(\vphi)|\bz|^2 dx=\io\ba{\mu}(\nabla\ba{\vphi}\cdot\bz) \, dx +\io\tau(\nabla\vphi\cdot\bz) \, dx+\io\mu(\nabla\zeta\cdot\bz) \dx+\io\Lm_1\cdot\bz \, dx,\label{F-diff-z}\\
          &\frac{L}{2}\frac{d}{dt}\|\zeta\|^2_V+L\io\nabla\tau\cdot\nabla\zeta \, dx+L\io\nabla\tau\cdot\nabla(-\Delta\zeta) \, dx=-L\io(\bz\cdot\nabla\varphi)\zeta -L\io(\ba{\bv}\cdot\nabla\ba{\vphi})\zeta \, dx\no\\&\quad-L\sigma\io(|\zeta|^2+|\nabla\zeta|^2) \, dx+L\io\Lm_2 \zeta\, dx-L\io(\bz\cdot\nabla\varphi)(-\Delta\zeta) -L\io(\ba{\bv}\cdot\nabla\ba{\vphi})(-\Delta\zeta) \, dx\no\\&\quad-L\io(\bv\cdot\nabla\zeta)(-\Delta\zeta)\, dx+L\io\Lm_2(-\Delta\zeta)\, dx\label{F-diff-zeta}\\
          &\text{as well as}\no\\
          &L\io\tau(-\Delta\zeta+\Delta^2\zeta-\tau)\, dx=L\io|\Delta^2\zeta|^2\, dx+L\io|\nabla\Delta\zeta|^2\, dx+\io(-\Delta\Lm_4)(-\Delta\zeta+\Delta^2\zeta-\tau) \, dx\no\\&\qquad+L\nu\io(-\Delta\zeta)(-\Delta\zeta+\Delta^2\zeta-\tau)\, dx+L\nu\io\Lm_4(-\Delta\zeta+\Delta^2\zeta-\tau)\, dx\no\\&\qquad+L\io f'(\vphi)(-\Delta\zeta)(-\Delta\zeta+\Delta^2\zeta-\tau)\, dx+L\io f'(\vphi)\Lm_4(-\Delta\zeta+\Delta^2\zeta-\tau)\, dx\no\\&\qquad+L\io\Lm_3(-\Delta\zeta+\Delta^2\zeta-\tau)\, dx,\label{F-diff-tau}
      \end{align}
   where $L$ is a positive constant whose value will be determined later. After adding the above three identities, some cancellation occur and the left hand side becomes 
     \begin{align}
         \al\|\bz\|^2_{\MV_{\td}}+\frac{L}{2}\frac{d}{dt}\|\zeta\|^2_V+L\io|\Delta^2\zeta|^2\, dx+L\io|\nabla\Delta\zeta|^2\, dx+L\io|\tau|^2\, dx,
     \end{align} where $\alpha = \min\{\eta, \lambda_\ast\}$, and we now estimate each term on the right-hand side. In doing so, we repeatedly apply several inequalities from \eqref{sobolev_ineq}--\eqref{GN-H3}, as well as H\"older, Young, and Poincar\'e inequalities, and occasionally integration by parts. Many of the terms are similar to those appearing in \eqref{test-w}--\eqref{test-tht}; therefore, we begin with the terms that are new and later briefly estimate the analogous ones. We start with the right-hand side of \eqref{F-diff-z} and estimate it using \eqref{cont-depend-est}, \eqref{est-R} and (A2) as follows:
     \begin{align}
         \io\Lm_1\cdot\bz \, dx=&\io(\lm(\vphi^u)-\lm(\vphi))\ba{\bv}\cdot\bz \, dx+\io(\lm'(\vphi)\zeta+R_\lm)\ba{\bv}\cdot\bz\, dx\no\\\leq &\|\vphi^u-\vphi\|_\infty\|\ba{\bv}\|\|\bz\|+\|\lm'(\vphi)\|_\infty\|\zeta\|_{L^4}\|\bv\|_{\ML^4}\|\bz\|+c\|\vphi^u-\vphi\|^2_\infty\|\bv\|\|\bz\|\no\\
         \leq & \frac{\al}{6}\|\nabla\bz\|^2+c_\al\|\ba{\vphi}\|^2_{H^2}\|\ba{\bv}\|^2+c_\al\|\bv\|^2_{\MV_{\td}}\|\zeta\|^2_V+c_\al\|\bv\|^2\|\ba{\vphi}\|^4_{H^2}.
     \end{align}
     Other three terms on the right hand side of \eqref{F-diff-z} can be estimated as follows
   \begin{align}
      \io\ba{\mu}(\nabla\ba{\vphi}\cdot\bz) \, dx &\leq \frac{\al}{6}\|\nabla\bz\|^2+c_\al\|\ba{\mu}\|^2\|\nabla\ba{\vphi}\|^2_{L^4}\leq \frac{\al}{6}\|\nabla\bz\|^2+c_\al\|\ba{\mu}\|^2\|\ba{\vphi}\|^2_{H^2},\\
      \io\tau(\nabla\vphi\cdot\bz) \, dx &\leq C_S\|\tau\|\|\nabla\vphi\|_{L^4}\|\nabla\bz\|
      \leq \frac{\al}{6}\|\nabla\bz\|^2+\frac{C_S^2\|\vphi\|^2_{H^2}}{4\al^2}\|\tau\|^2,\no\\
      \io\mu(\nabla\zeta\cdot\bz) \, dx&\leq \frac{\al}{6}\|\nabla\bz\|^2+c_\al\|\mu\|^2_V\|\nabla\zeta\|^2.\no
   \end{align}
   Next, we consider the right hand side of \eqref{F-diff-zeta}. Before that, let us calculate $\|\nabla R_g\|_\infty$. Recalling \eqref{tylor} we have
   \begin{align}
       \nabla R_g=&2\Big[\int_0^1(1-s)g''(s\vphi^u+(1-s)\vphi) \, ds\Big]\ba{\vphi}\nabla\ba{\vphi}\no\\&+\ba{\vphi}^2\Big[\int_0^1(1-s)g'''(s\vphi^u+(1-s)\vphi)(s\nabla\vphi^u+(1-s)\nabla\vphi) \, ds\Big].
   \end{align} owing to the uniform boundedness of both $\nabla\vphi^u$ and $\nabla\vphi$ from \eqref{unif_est_weak sol}, we get
   \begin{align}\label{est-grad-R}
       \|\nabla R_g\|_\infty\leq c\|\ba{\vphi}\|^2_{H^3}.
   \end{align}
  Then, using some estimates from \eqref{sobolev_ineq}--\eqref{GN-H3}, together with \eqref{tylor}, \eqref{est-R}, \eqref{cont-depend-est 3.34}, and \eqref{est-grad-R}, we obtain that
   \begin{align}
       L\io\Lm_2(\zeta-\Delta\zeta) \, dx &=L \io(h'(\vphi)\zeta+R_h)(\zeta-\Delta\zeta) \, dx\no\\&\leq L\|h'(\vphi)\|_\infty\|\zeta\|^2_V+\|R_h\|\|\zeta\|+\|\nabla R_h\|\|\nabla\zeta\|\no\\&\leq c\|\zeta\|^2_V+c\big(\|R_h\|^2_\infty+\|\nabla R_h\|^2_\infty\big)\no\\&\leq c\|\zeta\|^2_V+c\|\ba{\vphi}\|^4_{H^3}.
   \end{align}
 We now proceed to estimate the remaining terms in \eqref{F-diff-zeta}. Since similar terms have already appeared in \eqref{test-psi}–\eqref{test-del-psi} and were treated in detail in the previous subsection, we provide only brief estimates here.
   \begin{align*}
       &-L\io(\bz\cdot\nabla\varphi)\zeta \, dx \leq \frac{\al}{6}\|\nabla\bz\|^2+\frac{3L^2C_S}{2\al}\|\nabla\vphi\|^2_{L^4}\|\zeta\|^2\no\\& \qquad\qquad\qquad\qquad\quad
       \leq \frac{\al}{6}\|\nabla\bz\|^2+c_{\al, L}\|\nabla\vphi\|^2_{L^4}\|\zeta\|^2,\\
      &-L\io(\ba{\bv}\cdot\nabla\ba{\vphi})\zeta \, dx=L\io(\ba{\bv}\cdot\nabla\zeta)\ba{\vphi} \, dx \leq c_L\|\nabla\zeta\|^2+c_L\|\nabla\ba{\bv}\|^2\|\ba{\vphi}\|^2_V,\\
       &-L\sigma\io(|\zeta|^2+|\nabla\zeta|^2) \, dx\leq c_L\|\zeta\|^2_V,\\
       &-L\io(\bz\cdot\nabla\vphi)(-\Delta\zeta) \, dx\leq \frac{\al}{6}\|\nabla\bz\|^2+\delta\|\Delta^2\zeta\|^2+c_{\al,\delta,L}\|\vphi\|^4_{H^2}\|\zeta\|^2_V,\\
       &-L\io(\ba{\bv}\cdot\nabla\ba{\vphi})(-\Delta\zeta) \, dx\leq \delta\|\Delta\zeta\|^2+c_{\delta, L}\|\ba{\bv}\|^2_{\ML^4}\|\nabla\ba{\vphi}\|^2\no\\&\qquad\qquad\qquad\qquad\qquad\quad\leq \delta\|\Delta^2\zeta\|^2+c_{\delta}\|\zeta\|^2_V+c_{\delta, L}\|\ba{\bv}\|^2_{\ML^4}\|\nabla\ba{\vphi}\|^2,\\
        &-L\io(\bz\cdot\nabla\zeta)(-\Delta\zeta) \leq \delta\|\Delta^2\zeta\|^2+c_\delta\|\zeta\|^2_V+c_L\|\nabla\bv\|^2\|\nabla\zeta\|^2.
   \end{align*}
  To estimate the right-hand side of \eqref{F-diff-tau}, we consider a generic function $z\in L^2(0,T;H)$, since all such terms can be handled in a similar manner. As mentioned earlier, we first focus on the terms that did not appear in the previous estimates. Let us first compute
   \begin{align*}
       \Delta R_g=&2|\nabla\ba{\vphi}|^2\int_0^1(1-s)f''(s\vphi^u+(1-s)\vphi) \, ds+2\ba{\vphi}\Delta\ba{\vphi}\int_0^1(1-s)f''(s\vphi^u+(1-s)\vphi) \, ds\no\\&+4\ba{\vphi}\nabla\ba{\vphi}\int_0^1(1-s)f'''(s\vphi^u+(1-s)\vphi)(s\nabla\vphi^u+(1-s)\nabla\vphi) \, ds\no\\&+\ba{\vphi}^2\int_0^1(1-s)f^{(iv)}(s\vphi^u+(1-s)\vphi)|s\vphi^u+(1-s)\vphi|^2 \, ds\no\\&+\ba{\vphi}^2\int_0^1(1-s)f'''(s\vphi^u+(1-s)\vphi)(s\Delta\vphi^u+(1-s)\Delta\vphi) \, ds.
   \end{align*} 
   Using above identity we deduce that
   \begin{align*}
       \|\Delta R_g\|^2\leq & c\io\Big(|\nabla\ba{\vphi}|^4+|\ba{\vphi}|^2|\Delta\ba{\vphi}|^2+|\ba{\vphi}|^4(|\nabla\vphi^u|+|\nabla\vphi|)^2+|\ba{\vphi}|^4(|\Delta\vphi^u|+|\Delta\vphi|)^2\Big) \, dx\\
       \leq &\|\nabla\ba{\vphi}\|^4_{\ML^4}+\|\ba{\vphi}\|^2_{L^4}\|\Delta\ba{\vphi}\|^2_{L^4}+c\|\ba{\vphi}\|^4_\infty(\|\nabla\vphi^u\|^2+\|\nabla\vphi\|^2)+\|\ba{\vphi}\|^4_\infty(\|\Delta\vphi^u\|^2+\|\Delta\vphi\|^2).
   \end{align*} Thanks to \eqref{cont-depend-est} we have 
   \begin{align*}
       \|\Delta R_g\|^2_{L^2(0, T, H)}\leq &c\Big(\|\nabla\ba{\vphi}\|^4_{L^4(0, T; H^2)}+\|\ba{\vphi}\|^2_{L^\infty(), T; H^1)}\|\ba{\vphi}\|^2_{L^2(0, T; H^3)}\\&+\|\ba{\vphi}\|^4_{L^\infty(0, T; H^2)}(\|\vphi^u\|^2_{L^2(0, T; H^2)}+\|\vphi\|^2_{L^2(0, T; H^2)})\\&+\|\ba{\vphi}\|^4_{L^\infty(0, T; H^2)}(\|\vphi^u\|^2_{L^2(0, T; H^2)}+\|\vphi\|^2_{L^2(0, T; H^2)}\Big)\no\\\leq & K_4\Big(\|\vphi^u\|^2_{L^2(0, T; H^2)},\|\vphi\|^2_{L^2(0, T; H^2)}\Big)\|\bu\|^4_{L^2(0, T; \MH)}.
   \end{align*} 
   Then using above estimate on $R_f$ and \eqref{GN-H3}, we obtain
   \begin{align}\label{1st z}
       &-L\io(-\Delta\Lm_4)z \, dx=-L\io-\Delta(f'(\vphi)\zeta+R_f)z\, dx\no\\
&=L\io(f'''(\vphi)|\nabla\vphi|^2\zeta+f''(\vphi)\Delta\vphi\zeta+2f''(\vphi)\nabla\vphi\cdot\nabla\zeta+f'(\vphi)\Delta\zeta)z \, dx+L\io(\Delta R_f)z\, dx\no\\
       &\leq c_L(\|f'''(\vphi)\|_\infty\|\nabla\vphi\|^2_{\ML^8}\|\zeta\|_{L^4}+\|f''(\vphi)\|_\infty\|\Delta\vphi\|_{L^4}\|\zeta\|_{L^4}+\|f''(\vphi)\|_\infty\|\nabla\vphi\|_\infty\|\nabla\zeta\|\no\\&\qquad+\|f'(\vphi)\|_\infty\|\Delta\zeta\|)\|z\|+L\|\Delta R_f\|\|z\|\no\\
       &\leq \delta\|z\|^2+c_{ \delta,L}\|f'''(\vphi)\|^2_\infty\|\vphi\|^4_V\|\zeta\|^2_V+c_{\delta,L}\|f''(\vphi)\|^2_\infty\|\vphi\|^2_{H^3}\|\zeta\|^2_V\no\\&\qquad+c_{ \delta,L}\|f''(\vphi)\|^2_\infty\|\Delta\zeta\|^2+c_{\delta,L}\|\Delta R_f\|^2\no\\
       &\leq \delta\|z\|^2+\delta\|\Delta^2\zeta\|^2+c_{ \delta,L}\|f'''(\vphi)\|^2_\infty\|\vphi\|^4_V\|\zeta\|^2_V+c_{ \delta,L}\|f''(\vphi)\|^2_\infty\|\vphi\|^2_{H^3}\|\zeta\|^2_V\no\\&\qquad+c_{ \delta,L}\|f''(\vphi)\|^4_\infty\|\zeta\|^2_V+c_{\delta,L}\|\Delta R_f\|^2.
   \end{align}
  Next, we estimate the last term of \eqref{F-diff-tau}:
   \begin{align}\label{2nd z}
       -L\io(\Lm_3)z \, dx&=L\io(f'(\vphi^u)-f'(\vphi))\ba{w}z\, dx+L\io(f''(\vphi)\zeta+R_{f'})wz\, dx\no\\
       &\leq L\|\vphi^u-\vphi\|_{L^4}\|\ba{w}\|_{L^4}\|z\|+\|f''(\vphi)\|_\infty\|\zeta\|_{L^4}\|w\|_{L^4}\|z\|+\|\vphi^u-\vphi\|^2_\infty\|w\|\|z\|\no\\
       &\leq \delta\|z\|^2+c_{\delta, L}\|\ba{\vphi}\|^2_V\|\ba{w}\|^2_V+c_{\delta, L}\|\ba{\vphi}\|^4_{H^2}\|w\|^2+c_{\delta, L}C_S\|f''(\vphi)\|^2_\infty\|w\|^2_V\|\zeta\|^2_V.
   \end{align}
   The remaining terms on the right-hand side of \eqref{F-diff-tau} are estimated in the same way, using \eqref{est-R} and \eqref{sobolev_ineq}–\eqref{GN-H3}.
   \begin{align}
       &-L\nu\io(-\Delta\zeta) z \, dx\leq \delta\|z\|^2+c_{\delta, L}\|\Delta\zeta\|^2\leq \delta\|z\|^2+\delta\|\Delta^2\zeta\|^2+\frac{C_S^2c_{\delta, L}^2}{\delta}\|\zeta\|^2_V,\label{3rd z}\\
       &-L\nu\io(\Lm_4)z\, dx=-L\nu\io(f'(\vphi)\zeta+R_f)z \, dx\leq \delta\|z\|^2+c_{\delta, L}\|f'(\vphi)\|^2_\infty\|\zeta\|^2+c_{\delta, L}\|\ba{\vphi}\|^4_\infty,\label{4th z}\\
       &-L\io f'(\vphi)(-\Delta\zeta)z \, dx\leq \delta\|z\|^2+\delta\|\Delta^2\zeta\|^2+\frac{C^2_Sc^2_{L,\delta}\|f'(\vphi)\|^4_\infty}{\delta}\|\zeta\|^2_V,\label{5th z}\\
       &-L\io f'(\vphi)\Lm_4 z\, dx\leq\delta\|z\|^2+c_{\delta, L}\|f'(\vphi)\|^4_\infty\|\zeta\|^2+c_{\delta, L}\|f'(\vphi)\|^2_\infty\|\ba{\vphi}\|^4_\infty.\label{6th z}
        \end{align}
        Taking $z=-\Delta\zeta+\Delta^2\zeta-\tau$ in \eqref{1st z}–\eqref{6th z} and estimating as in \eqref{3term expan}, we obtain the bounds for all terms on the right-hand side after adding \eqref{F-diff-z}–\eqref{F-diff-tau}. Finally, we observe the following points:
        \begin{enumerate}
        \item Since $\vphi\in L^\infty(Q)$ and $f$ satisfies the assumption(A4)-(A5), we have $\sum_{k=1}^5\|f^{(k)}(\vphi)\|_{L^\infty(Q)} \leq C$, for some positive constant $C.$
            \item Noting that $\ms(\bg)=(\bv,\vphi,\mu,w)$ is a solution of \eqref{CHB}, it satisfies Theorem \ref{CHB-wellposed}. Hence, from \eqref{unif_est_weak sol} we obtain
\begin{align*}\|\vphi(t)\|_{W}\leq K_1,\quad \text{for a.e. } t\in(0,T).\end{align*}
 \item Let us denote 
            \begin{align}
            \qquad \Phi_1(t)=&\|\bv(t)\|^2_{\MV_{\td}}+\|\mu(t)\|^2_V+\|\vphi(t)\|^2_{H^2}+\|\vphi(t)\|^4_{H^2}+\|f'''(\vphi(t))\|^2_\infty\|\vphi(t)\|^4_V\no\\&+\|f''(\vphi(t))\|^2_\infty(\|w(t)\|^2_V+\|\vphi(t)\|^2_{H^3})+\|f''(\vphi(t))\|^4_\infty+\|f'(\vphi(t))\|^2_\infty+\|f'(\vphi(t))\|^4_\infty+1,\no\\
            \Phi_2(t)=&\|\bv(t)\|^2+\|w(t)\|^2 +\|f'(\vphi(t))\|^2_\infty+\|\vphi^u(t)\|_{H^2} +\|\vphi(t)\|^2_{H^2}+1.\no
            \end{align}
             Then, using Theorem \ref{CHB-wellposed}, we have $\Phi_1$ in $L^1(0, T)$ and $\Phi_2$ in $L^\infty(0, T).$
       \end{enumerate}
        At this stage, gathering all the inequalities derived above and taking into account the preceding observations, we combine them into the sum appearing in \eqref{F-diff-z}–\eqref{F-diff-tau}. After a suitable rearrangement, we obtain
        \begin{align}
           &\frac{L}{2}\|\zeta(t)\|^2_V+\frac{\al}{2}\int_0^t\|\bz(s)\|^2_{\MV_{\td}} ds+(L-26\delta)\int_0^t\|\Delta^2\zeta(s)\|^2 ds +L\int_0^t\|\nabla\Delta\zeta(s)\|^2 ds\no\\&\quad+\Big(L-15\delta-\frac{C^2_SK^2_1}{4\al^2}\Big)\int_0^t\|\tau(s)\|^2 ds\leq C\int_0^t\Phi_1(s)\|\zeta(s)\|^2_V ds\no\\&\quad+C\int_0^t\Phi_2(s)\left(\|\ba{\vphi}(s)\|^4_{H^2}+\|\ba{\vphi}(s)\|^2_V\|\ba{w}(s)\|^2_V+\|\ba{\vphi}(s)\|^2_V\|\ba{\bv}(s)\|^2_{\MV_{\td}}+\|\ba{\vphi}(s)\|^4_{H^3}+\|\ba{\mu}(s)\|^2_V\|\ba{\vphi}(s)\|^2_V\right) ds,\no
        \end{align} 
      for every $t\in[0,T]$, where the constant $C$ in the above inequality depends on $L$, $\al$, $\delta$, and the Sobolev embedding constant $C_S$.  Choosing $\delta$ and $L$ appropriately (for instance, $L>\frac{13C_S^2K_1^2}{37\al^2}$ and $\delta=\frac{L}{52}$) and applying the generalized Gronwall inequality [cf. Lemma \ref{lem:Gronwall}], we obtain

        \begin{align}\label{z-xi-tau-bd}
            \|\bz\|_{L^2(0, T; \MV_{\td})}+\|\zeta\|_{L^\infty(0, T, V)\cap L^2(0, T; H^4)}+\|\tau\|_{L^2(0, T; H)}\leq K_5\|\bu\|^2_{L^2(0, T; \MH).}
        \end{align} 
       Now, using the expression of $\gm$ in \eqref{gm equ}, together with \eqref{z-xi-tau-bd} and an estimate for $\Lambda_4$ analogous to that in \eqref{1st z}, and then invoking elliptic regularity theory, it follows that
\begin{align}\label{gm-bd}
        \|\gm\|_{L^2(0, T; W)} \leq K_5 \|\bu\|^2_{L^2(0, T; \MH)}.
\end{align}
Using the above two estimates in conjunction with \eqref{zeta_equ}, we obtain the following estimate for the time derivative of $\zeta$:
\begin{align}\label{t-zeta-bd}
            \|\pat\zeta\|_{L^2(0, T; W')}\leq K_5\|\bu\|^2_{L^2(0, T; \MH)}. 
        \end{align}
        Now adding the inequalities \eqref{z-xi-tau-bd}-\eqref{t-zeta-bd} we get \eqref{Diff_est}, which conclude the proof.
        \end{proof}
\begin{remark}
    The inequality \eqref{Diff_est} also shows that the map 
\[
\tilde{\ms}:\U \to \W \times L^2(0, T; H) \times L^2(0, T; W),
\qquad 
\tilde{\ms}(\bg) = (\bv, \vphi, \mu, w),
\]
is Fr\'echet differentiable, where $(\bv, \vphi, \mu, w)$ denotes the solution of the state system \eqref{CHB} corresponding to the control $\bg$. 
This implies that we may additionally include the terms 
\[
\frac{\beta_5}{2}\int_Q |\mu - \mu_Q|^2 
\;+\; 
\frac{\beta_6}{2}\int_Q |w - w_Q|^2,
\]
with target functions $\mu_Q,\, w_Q \in L^2(Q)$ and coefficients $\beta_5,\, \beta_6 \ge 0$, in the cost functional \eqref{cost functional}.  
The entire analysis carries over with only minor modifications. 
\end{remark}
\section{The optimal control problem.}
In this section, we study the existence of an optimal control problem $\textbf{(CP)}_\eta$ with the cost functional defined in \eqref{cost functional} and characterized it using first order necessary optimality condition. 
Note that, if $\kappa=0$, the cost functional $\J=J$ is differentiable, and the optimality system together with the first-order necessary optimality condition can be derived by the Lagrange multiplier method; see \cite[Chapter 2, Theorem 1.9]{fursikov_book}. 
However, when $\kappa\neq 0$, this approach is not applicable, and we instead rely on \cite[Theorem 3.2]{CT20} to obtain the optimality condition. 
Thus, we consider the two cases $\kappa=0$ and $\kappa>0$ separately. 
We begin by proving the existence of a solution to the control problem $\textbf{(CP)}_\eta$.

\subsection{Existence of an optimal control.} 
We note from the operator $\ms$ that control can uniquely define the solution $(\bv,\vphi, \mu, w)$ of the system \eqref{CHB}. Then the cost functional given in \eqref{cost functional} can be defined in terms of control only.  Therefore, now we define the reduced cost-functional $ \Tilde{\mathcal{J}} : \mathcal{U} \rightarrow [0, \infty)$ as 
\begin{align}\label{reformulated}
     \Tilde{\mathcal{J}}(\bg) := \mathcal{J}((\bv, \vphi), \bg)= \mathcal{J}(\mathcal{S}(\bg), \bg);=\tilde{J}(\bg)+\kappa j(\bg).
\end{align}
In the following theorem we show that the optimal control problem $\textbf{(CP)}_\eta$ admits at least one solution.
\begin{theorem}\label{exist-op-control}
    There exist at least one solution for the control problem $\textbf{(CP)}_\eta$. That is, there exists a $\bg^\ast\in\U_{ad}$ such that 
    $$\J(\ms(\bg^\ast), \bg^\ast)\leq \J(\ms(\bg), \bg) \quad \text{ for every }\bg\in\U_{ad}.$$
\end{theorem}
\begin{proof}
    From \eqref{reformulated}, we can write the optimal control problem $\textbf{(CP)}_\eta$ as $$\min_{\bg\in\mathcal{U}_{ad}} \Tilde{\J}(\bg).$$ 
First note that $\Tilde{\J}$ is bounded from below and nonnegative as $\beta_i\geq 0$, for $i=1, \cdots, 4$. As a consequence, the infimum exists and finite, so, let us take $j = \inf_{\bg \in \U_{ad}} \Tilde{\J}(\bg)$.  Also, since $ \U_{ad}$ is nonempty, there exists a minimizing sequence $\bg_n \in \mathcal{U}_{ad}$ such that
\begin{align*}
    \lim_{n \rightarrow \infty} \Tilde{\J}(\bg_n) = j.
\end{align*} 
Thanks to  the definition of $\U_{ad}$, there exists a subsequence $\{\bg_{n_k}\}$ (still denoted by $\{\bg_n\}$) such that
$$\bg_n\to\bg^\ast \quad \text{ weakly 
in } L^2(0, T; \MH),$$
for some $\bg^\ast\in\U_{ad}$, as $\U_{ad}$ is weakly closed subset of $\U$  (cf. \cite[Proposition 5.1.4]{K23}). Now, for the weakly convergent subsequence $\{\bg_n\}$ we consider $\ms(\bg_n)=(\bv_n,\vphi_n),$ where $\bv_n,\vphi_n$ are the first two component of the solution state \eqref{CHB}-\eqref{in CHB} corresponding to control $\bg_n.$  Then, from the stability estimate \eqref{cont-depend-est}, we obtain a subsequence of $(\bv_{n_k}, \vphi_{n_k}, \mu_{n_k}, w_{n_k})_{k\in\mathbb{N}}$ (still denoted by $(\bv_n, \varphi_n, \mu_n, w_n)$) such that
\begin{equation}
    \left\{
\begin{aligned}
    &\bv_n\to\bv^\ast \quad\text{ weakly in } L^2(0, T; \MV_{\td}),\\
    &\vphi_n\to \vphi^\ast\quad\text{ weakly in } C^0([0, T]; V)\cap L^2(0, T; H^4),\\
    &\mu_n\to\mu^\ast\quad\text{ weakly in } L^2(0, T; H),\\
    &w_n\to w^\ast\quad\text{ weakly in } L^2(0,T; W).
\end{aligned}
\right.
\end{equation}
Above convergence property of $\vphi_n$ and well-known strong convergence results (cf. \cite[Section 8, Corollary 4]{JSimon} imply $$\vphi_n\to\vphi^\ast \quad\text{ strongly in }L^\infty(Q).$$ Using these convergence results on $(\bv_n, \vphi_n, \mu_n, w_n)$, we can easily pass to the limit in the weak formulation of \eqref{CHB}-\eqref{in CHB} (see \eqref{v weak form}-\eqref{w weak form}) and conclude $\ms(\bg^\ast)=(\bv^\ast, \vphi^\ast).$\\
Since $\Tilde{\J}$ is weakly lower semi-continuous on $\mathcal{U}_c$, we have (see \cite[Corollary 5.1.5]{K23})
\begin{align*}
   \Tilde{\J}(\bg^\ast) \leq \liminf_{n \rightarrow \infty}\Tilde{\J}(\bg_n)=j.
\end{align*}
Therefore, $j \leq \Tilde{\J}(\bg^\ast) \leq  \liminf_{n \rightarrow \infty}\Tilde{\J}(\bg_n) = j,$\\
which yields that $ \bg^\ast$ is a solution of the problem $\textbf{(CP)}_\eta$.
\end{proof}
\begin{definition}\label{opt_sol}
  We call $\bg^\ast $ to be an optimal control and the corresponding solution of the system \eqref{CHB} denoted by $((\bv^\ast, \varphi^\ast, \mu^\ast, w^\ast), \bg^\ast) $ is refer as the optimal solution of the problem $\textbf{(CP)}_\eta$.
\end{definition}
\subsection{The optimality system.} In this subsection we particularly take $\kappa=0$. Having prove existence of optimal solution, we now use Lagrange multiplier principle to obtain an optimality system that an optimal solution and Lagrange multiplier should satisfy. For that,  let us recall abstract  Lagrange multipliers method from \cite[Theorem 1, Theorem 3, Chapter 1]{IT79} and \cite{fursikov_book}.\\
Let $X_1,$ $X_2$ be two Banach spaces. Let $f:X_1\to \mathbb{R}$ and $g:X_1\to \mathbb{R}$ be two functionals and $G:X_1\to X_2$ be a mapping. We find a $z\in X_1$ such that
\begin{align}\label{min}
    & f(z)=\inf_{u\in\mathcal{Z}_{ad}}f(u),\\
    \text{where } & \mathcal{Z}_{ad}=\{u\in X_1: G(u)=0, g(u)\leq 0\}.\no
\end{align}
The Lagrange functional for the minimization problem \eqref{min} is defined by 
\begin{align}\label{abs_L}
    \mathcal{L}(z, \lambda_0,\lambda,q)=\lambda_0f(z)+\langle G(z), q\rangle+\lambda g(z)
\end{align}
for all $z\in X_1, \lambda_0, \lambda\in\mathbb{R},$ and $q\in X_2',$ where $\langle\cdot,\cdot\rangle$ denotes the duality pairing between $X_2$ and $X_2'.$ Then we have following abstract Lagrange principle:
\begin{theorem}\label{ALMP}(\cite[Chapter 2, Theorem 1.9]{fursikov_book}) Let $z$ be a solution of \eqref{min}. Assume that the mapping $f, g, G$ are continuously differentiable and that the image of the operator $G'(z):X_1\to X_2$ is closed. Then there exist  $q\in X_2'$ and $\lambda_0, \lambda\in\mathbb{R}$ such that the triplet $(q, \lambda_0,\lambda)\neq (0,0,0)$. Moreover, it satisfies the relations
\begin{align*}
    \langle\mathcal{L}_z(z, \lambda_0, \lambda,q), h\rangle=0 \quad \forall \, h \, \in \, X_1,\\
    \lambda_0, \lambda\geq 0, \text{ and }\lambda g(z)=0,
\end{align*}
 where $\mathcal{L}_z(\cdot,\cdot,\cdot,\cdot)$ denotes the Fr\'echet derivative of $\mathcal{L}$ with respect to first variable. Furthermore, if $G'(z):X_1\to X_2$ is an epimorphism and the constraint $g(z)\leq 0$ is absent in \eqref{min}, then  $\lambda_0\neq 0$ and can be taken as 1.  
\end{theorem}
We now formulate the above theoretical framework in the context of our problem. To this end, we consider the following functional spaces:
\begin{align}
 &X_1:=\widetilde{\W}\times\U, \quad X_2:= L^2(0, T; \MV_{\td}')\times L^2(0, T; W')\times L^2(0, T; H)\times L^2(0, T; W'),\label{X2-norm}\\
 &\text{where }\widetilde{\W}:= \{(\bv, \psi)\in\W:\psi(0)=0, \, \text{ in }\Omega\}\times L^2(0, T; H)\times L^2(0, T; W).\no
\end{align}
We also denote \begin{align}\label{short-notation}\vk=(\bv, \vphi, \mu, w), \quad \vk^\ast=(\bv^\ast, \vphi^\ast, \mu^\ast, w^\ast), \quad \ell=(\bw, \psi, \tht, \xi), \, \text{ and }\,  \wp=(\bv^a, \vphi^a, \mu^a, w^a)\end{align}. 
Let $\mathbb{P}$ be the Leray projector defined as $\mathbb{P}:\MH^1(\Omega)'\to\MV_{\td}'$, and consider the mappings $f:X_1\to\mathbb{R}$ and $G:X_1\to X_2$ given by:
\begin{equation}\label{def-G}
\left\{
 \begin{aligned}
&f( \vk, \bg)=J(\vk, \bg),\\
&G\begin{pmatrix}
    \vk\\ \bg
\end{pmatrix}=\begin{pmatrix}
    \mathbb{P}[- \eta\Delta\bv+ \lm(\vphi)\bv-\mu\nabla\varphi-\bg]\\
    \pat\vphi +\bv\cdot \nabla \vphi-\Delta\mu+\sigma\vphi-h(\vphi)\\ -\Delta w+f'(\vphi)w+\nu w-\mu\\ -\Delta\vphi+f(\vphi)-w
\end{pmatrix}.
\end{aligned}
 \right.
 \end{equation}
Then we reformulate the optimal control problem $\textbf{(CP)}_\eta$ as follows
\begin{equation}\label{ref_CP}
\left\{
\begin{aligned}
    &f(\vk^\ast, \bg^\ast)=\inf_{(\vk, \bg)\in \W_{ad}}J(\vk, \bg),\\
    &\W_{ad}=\{\bg\in X_1: G((\vk, \bg)^t)=\mathbf{0}, \|\bg\|^2_{L^2(0, T; \MH)}-M^2\leq 0\}.
\end{aligned}
\right.
\end{equation}
The Lagrange functional of the extremal problem \eqref{ref_CP} is define by the formula
\begin{align}\label{Lf}
    \mathcal{L}(\vk, \bg, \lm_0, \lm_1, \wp)=\lm_0 f(\vk, \bg)+\langle G((\vk, \bg)^t), \wp^t\rangle_{X_2, X_2'}+\lm_1(\|\bg\|^2_{L^2(0, T; \MH)}-M^2).
\end{align}
Now we verify the hypothesis of Theorem \ref{ALMP}. From Theorem \ref{exist-op-control}, we conclude that there exists a solution $(\vk^\ast, \bg^\ast)\in\W_{ad}$ of the problem \eqref{ref_CP}.  Clearly, $f$ is continuously differentiable. Let $(\vk_n, \bg_n)$ be a sequence in $X_1$ such that $(\vk_n, \bg_n) \to (\vk, \bg)$ in $X_1$. Then, using \eqref{cont-depend-est}, we can conclude that $G((\vk_n,\bg_n)^t)\to G((\vk, \bg)^t)$ in $X_2$ (to conclude this convergence in $X_2$, \eqref{cont-depend-est} is a stronger estimate). The Fr\'echet derivative of $G$ with respect to $(\vk, \bg)$ evaluated at $(\vk^\ast, \bg^\ast)$, $G'((\vk^\ast, \bg^\ast)^t)\in\mathcal{L}(X_1, X_2)$ is given by
\begin{align}\label{der-G}
    {\small G'\begin{pmatrix}
        \vk^\ast\\\bg^\ast
    \end{pmatrix}\begin{pmatrix}
        \ell\\\bu
    \end{pmatrix}=\begin{pmatrix}
   \mathbb{P}(-\eta\Delta\bw+\lm'(\vphi^\ast)\psi\bv^\ast+\lm(\vphi^\ast)\bw-\tht\nabla\vphi^\ast-\mu^\ast\nabla\psi-\bu)\\
   \pat\psi+\bw\cdot\nabla\vphi^\ast+\bv^\ast\cdot\nabla\psi-\Delta\tht+\sigma\psi-h'(\vphi^\ast)\psi \\
      -\Delta\xi+f''(\vphi^\ast)\psi w^\ast+f'(\vphi^\ast)\xi+\nu\xi-\tht\\
      -\Delta\psi+f'(\vphi^\ast)\psi-\xi 
  \end{pmatrix}} \quad\text{ for every }(\ell, \bu)\in X_1.\end{align}
  It is immediate from the estimate \eqref{lin-energy} that $G'\begin{pmatrix}
        \vk^\ast\\\bg^\ast
    \end{pmatrix}\begin{pmatrix}
        \ell_n\\\bu_n
    \end{pmatrix}\to G'\begin{pmatrix}
        \vk^\ast\\\bg^\ast
    \end{pmatrix}\begin{pmatrix}
        \ell\\\bu
    \end{pmatrix}$ in $X_2$ whenever $(\ell_n, \bu_n)\to (\ell, \bu)$ in $X_1.$
  Now to show that $G'\begin{pmatrix}
        \vk^\ast\\\bg^\ast
    \end{pmatrix}$ is onto, we need to show for every  $(\bg_1, g_2, g_3, g_4)\in X_2$, there exists $(\ell, \bu)\in X_1$ such that the system \begin{align}\label{epi-equ}
     \small G'\begin{pmatrix}
        \vk^\ast\\\bg^\ast
    \end{pmatrix}\begin{pmatrix}
        \ell\\\bu
    \end{pmatrix}=(\bg_1, g_2, g_3, g_4)^t,
  \end{align}
  has a solution. Indeed, writing the above equation in variational form we get 
  \begin{align}
     &\eta\int_\Omega\nabla\bw :\nabla\bz dx+\int_\Omega\lm'(\vphi^\ast)\psi\bv^\ast\cdot\bz dx +\int_\Omega\lm(\vphi^\ast)\bw\cdot\bz \dx=\int_\Omega\tht\nabla\vphi^\ast\cdot\bz dx+\int_\Omega\mu^\ast\nabla\psi\cdot\bz dx+\int_\Omega\bu\cdot\bz dx\no\\&\qquad+\int_\Omega\bg_1\cdot\bz dx,\text{ for every }\bz\in\MV_{\td},\label{1st-equ}\\
     &\int_\Omega\pat\psi\rho dx+\int_\Omega (\bw\cdot\nabla\vphi^\ast)\rho dx+\int_\Omega(\bv^\ast\cdot\nabla\psi)\cdot\rho dx -\int_\Omega\tht\Delta\rho dx=-\int_\Omega\sigma\psi\rho dx+\int_\Omega h'(\vphi^\ast)\psi\rho dx\no\\&\qquad+\io g_2\rho dx \text{ for every }\rho\in W,\\
     &\int_\Omega\nabla\xi\cdot\nabla\rho dx +\int_\Omega f''(\vphi^\ast)\psi w^\ast\rho dx+\int_\Omega f'(\vphi^\ast)\xi\rho dx+\nu\int_\Omega\xi\rho dx=\int_\Omega\tht\rho dx+\io g_3\rho dx \text{ for every } \rho \in V,\\
     &\int_\Omega\nabla\psi\cdot\nabla\rho dx+\int_\Omega f'(\vphi)\psi\rho dx=\int_\Omega\xi\rho dx+\io g_4 \rho dx  \text{ for every } \rho \in W.
 \end{align}
We supplement this system with the initial and boundary conditions  
\begin{align}\label{last equ}
    \psi|_{t=0}=0 \text{ in } V, 
    \qquad 
    \pan\psi = \pan\tht = \pan\xi = 0 \text{ on } \Sigma.
\end{align}
Using Proposition \ref{exist_lin_sys}, we conclude that the system \eqref{1st-equ}--\eqref{last equ} admits a solution. 
Indeed, although the system contains additional terms involving $\bg_1$ and $g_i$, $i=2, 3, 4$, these can be estimated by means of H\"older's and Young's inequalities.  
Thus, all the hypotheses of Theorem \ref{ALMP} are verified, and therefore there exists $\wp \in X_2'$ such that
 \begin{align}\label{derivative_LP}
     \mathcal{L}_{\vk}(\vk^\ast, \bg^\ast, \lm_0, \lm_1, \wp)=0, 
 \end{align}
where $\mathcal{L}_{\vk}$ denotes the Fr\'echet derivative of $\mathcal{L}$ with respect to its first variable. Gathering the preceding arguments in Theorem \ref{ALMP}, we arrive at the following theorem
\begin{theorem}\label{exist_LM}
    Let $(\vk^\ast, \bg^\ast)$ be a solution of the optimal control problem $\textbf{(CP)}_\eta$. Then there exists $\wp\in X_2'$ ($\wp$ is define in \eqref{short-notation}) such that 
    \begin{align}\label{LP}
       & \lm_0\left(\beta_1\int_Q(\bv^\ast-\bv_Q)\cdot\bw\, dxdt +\beta_2\int_Q(\vphi^\ast-\vphi_Q)\psi\, dxdt +\beta_3\io(\vphi^\ast(T)-\vphi_T)\psi(T)\, dx +\beta_4\int_Q\bg^\ast\cdot\bu\, dx dt\right)\no\\
       &+\int_Q\left[-\eta\Delta\bw+\lm'(\vphi^\ast)\psi\bv^\ast+\lm(\vphi^\ast)\bw-\tht\nabla\vphi^\ast-\mu^\ast\nabla\psi\right]\cdot\bv^a\no\\&+\int_Q\left[ \pat\psi+\bw\cdot\nabla\vphi^\ast+\bv^\ast\cdot\nabla\psi-\Delta\tht+\sigma\psi-h'(\vphi^\ast)\psi\right]\vphi^a+\no\\&+\int_Q\left[ -\tht-\Delta\xi+f''(\vphi^\ast)\psi w^\ast+f'(\vphi^\ast)\xi+\nu\xi\right]\mu^a+\int_Q\left[-\xi-\Delta\psi+f'(\vphi^\ast)\psi \right]w^a=0, \qquad,
    \end{align}
  for every $(\ell, \bu)\in X_1$,   and 
    \begin{align}\label{const1}
        \lm_1\left(\int_Q|\bg^\ast|^2 dx dt-M^2\right)=0,
    \end{align}
    for some $ \lm_1\geq 0$ and $\lm_0\neq 0.$
\end{theorem}
\begin{proof}
The identity \eqref{LP} follows from expanding \eqref{derivative_LP} by means of the chain rule. 
To complete the proof of the theorem, it remains to show that $\lambda_0 \neq 0$. 
To this end, we follow the approach of \cite{Fursikov_bndry}. 
By applying integration by parts in \eqref{LP} and varying $(\ell, \bu) \in X_1$, we obtain the optimality system, which coincides with the adjoint system, as follows:
    \begin{equation}\label{adj sys}
        \left\{
        \begin{aligned}
        &-\eta\Delta\bv^a+\lm(\vphi^\ast)\bv^a+\vphi^a\nabla\vphi^\ast=-\lm_0\beta_1(\bv^\ast-\bv_Q) \text{ in } Q,\\
        &-\pat\vphi^a+\lm'(\vphi^\ast)\bv^\ast\bv^a+\bv^a\cdot\nabla\mu^\ast-\bv^\ast\cdot\nabla\vphi^a+\sigma\vphi^a-h'(\vphi^\ast)\vphi^a
        +f''(\vphi^\ast)w^\ast\mu^a-\Delta w^a\\&\qquad+f'(\vphi^\ast)w^a=-\lm_0\beta_2(\vphi^\ast-\vphi_Q) \text{ in } Q,\\
        &\mu^a=-\Delta\vphi^a-\bv^a\cdot\nabla\vphi^\ast \text{ in } Q,\\
        & w^a=-\Delta\mu^a+f'(\vphi^\ast)\mu^a+\nu\mu^a \text{ in } Q,\\
       & \td{\bv}^a=0\text{ in } Q,\\
       & \bv^a=0, \pan\vphi^a=\pan\mu^a=\pan w^a=0 \text{ on }\Sigma,\\
       &\vphi^a(T)=-\lm_0\beta_3(\vphi^\ast(T)-\vphi_T) \text{ in } \Omega. 
        \end{aligned}
        \right.
    \end{equation} 
    If $\lm_0=0$, using standard energy estimates and Gronwall inequality we can show that the system \eqref{adj sys} has zero solution. Then we want $\lm_1\neq 0$, otherwise $(\wp, \lm_0, \lm_1)=(\mathbf{0}, 0, 0)$, which will contradict Theorem \ref{ALMP}, thus $\lm_1>0.$ By virtue of \eqref{const1}, we consider the following modified minimization problem: 
    $$\text{minimize $J$ subject to equality constraint \eqref{CHB} and \eqref{const1}}$$. Now we show that this modified minimization problem also satisfied the condition of Theorem \ref{ALMP}. We set $X_1$ is same as before and $\tilde{X}_2$ as $X_2\times\rea$. Now define the mappings  \begin{equation}
\left\{
 \begin{aligned}
&f( \vk, \bg)=J(\vk, \bg),\\
&\tilde{G}\begin{pmatrix}
    \vk\\\bg
\end{pmatrix}=\begin{pmatrix}
    \mathbb{P}[- \eta\Delta\bv+ \lm(\vphi)\bv-\mu\nabla\varphi-\bg]\\
    \pat\vphi +\bv\cdot \nabla \vphi-\Delta\mu+\sigma\vphi-h(\vphi)\\ -\Delta w+f'(\vphi)w+\nu w-\mu\\ -\Delta\vphi+f(\vphi)-w\\
    \int_Q|\bg|^2 \, dx dt-M^2
\end{pmatrix},
\end{aligned}
 \right.
 \end{equation}
  where $J$ is same as in \eqref{cost functional}. Then $\tilde{G}'
      (\vk^\ast, \bg^\ast):X_1\to X_2$  define by
  \begin{align}
    \tilde{G}'\begin{pmatrix}
        \vk^\ast\\ \bg^\ast
    \end{pmatrix}\begin{pmatrix}
        \ell\\\bu 
    \end{pmatrix}=\begin{pmatrix}
   \mathbb{P}(-\eta\Delta\bw+\lm'(\vphi^\ast)\psi\bv^\ast+\lm(\vphi^\ast)\bw-\tht\nabla\vphi^\ast-\mu^\ast\nabla\psi-\bu)\\
   \pat\psi+\bw\cdot\nabla\vphi^\ast+\bv^\ast\cdot\nabla\psi-\Delta\tht+\sigma\psi-h'(\vphi^\ast)\psi \\
      -\Delta\xi+f''(\vphi^\ast)\psi w^\ast+f'(\vphi^\ast)\xi+\nu\xi-\tht\\
      -\Delta\psi+f'(\vphi^\ast)\psi-\xi\\
      \int_Q\bg^\ast\cdot\bu \, dx dt,
  \end{pmatrix} 
  \end{align}
for every $(\ell, \bu)\in X_1$.  To show $\tilde{G}'((\vk^\ast, \bg^\ast)^t)$ is an epimorphism, we first observe that the operator is continuous. Next we need to show for each $(\bg_1, g_2, g_3, g_4, \zeta)\in \tilde{X}_2$, there exists $(\ell, \bu)\in X_1$ that satisfy 
  \begin{align}\label{last_const}
     \tilde{G}'\begin{pmatrix}
        \vk^\ast\\ \bg^\ast
    \end{pmatrix}\begin{pmatrix}
        \ell\\\bu 
    \end{pmatrix}=(\bg_1, g_2, g_3, g_4)^t, \quad \int_Q\bg^\ast\cdot\bu \, dx dt=\zeta.
  \end{align}
  We have already seen that first expression of \eqref{last_const} has a solution say $\tilde{\ell}\in \W$. It is sufficient to show that there exists a $\bu \in \U$ for which the left-hand side of \eqref{last_const} is not equal to zero, for then we can obtain second equality of \eqref{last_const} by multiplying $\bu$ by a suitable constant. Suppose that for every $\bu \in \U$
$$ \int_Q \bg^\ast \cdot \bu \, dx\, dt = 0 $$
holds, which implies that $\bg^\ast(x,t)=0$ for a.e. $(x,t)\in Q$, contradicting \eqref{const1}. Hence $(\tilde{\ell}, \tilde{\bu})\in X_1$ satisfy \eqref{last_const} implying $\tilde{G}'((\vk^\ast, \bg^\ast)^t)$ is an epimorphism.
By virtue of Theorem \ref{ALMP}, there exist Lagrange multipliers triplet $(\tilde{\wp}, \tilde{\lm}_0, \tilde{\lm}_1)$ with $\tilde{\wp}\in \tilde{X}'_2,$  $\tilde{\lm}_0 \neq 0$ and $\tilde{\lm}_1\geq 0 $ such that \eqref{derivative_LP} holds, where $\mathcal{L}$ is defined in \eqref{Lf}. In particular, $(\wp, \lm_0, \lm_1)$ is such a triplet, which contradicts the assumption $\lm_0 = 0$, hence $\lm_0 \neq 0$.
\end{proof}
Now onwards we will consider $\lm_0=1$ in the adjoint system \eqref{adj sys}.
\begin{remark}
    We note that if $(\bv^a, \vphi^a, \mu^a, w^a)$ satisfy the system \eqref{adj sys} with the above mention regularity, particularly, $\bv^a\in L^2(0, T; \MV_{\td})$, then by means of De Rham's lemma (cf. \cite{Temam}) we can recover the pressure term in the first equation of \eqref{adj sys}. In particular, there exist $p^a\in L^2(0, T; H)$ such that 
    \begin{align}
        -\eta\Delta\bv^a+\lm(\vphi^\ast)\bv^a+\vphi^a\nabla\vphi^\ast-\nabla p^a=-\beta_1(\bv^\ast-\bv_Q) \text{ in } Q,
    \end{align}
    holds in the sense of distribution.
\end{remark}
 With the adjoint system \eqref{adj sys} in hand paves the way to characterized the optimal control and deriving the first order necessary optimality condition.
 \subsection{First order necessary optimality condition for the case $\kappa=0$}
 By virtue of Fr\'echet differentiability given in Theorem \ref{thm-f-diff} and quadratic structure of the cost functional $J$ given in \eqref{cost functional}, we can apply chain rule to the composition map
 $$\U\to \W\to\rea \text{ given by }\bg\mapsto (\bv, \vphi, \bg):=(\ms(\bg), \bg)\mapsto J(\ms(\bg), \bg):=\tilde{J}(\bg).$$  By chain rule, we can write
 \begin{align}
D\tilde{J}(\bg)= DJ(S(\bg), \bg) = J'_{\ms(\bg)}(\ms(\bg), \bg) \circ \ms'(\bg) + J'_{\bg}(\ms(\bg), \bg)\circ \mathrm{Id}.\label{chain rule}
  \end{align}
  In the above equation \eqref{chain rule}, we have for a fixed $\bg \in \mathcal{U}$, the Gateaux derivative of $J(\ms(\bg), \bg)$ with respect to $\mathcal{S}(\bg) = (\bv,\varphi)$ and $\bg$ are denoted by $J'_{\mathcal{S}(\bg)}$ and $J'_{\bg}$, respectively. The Gateaux derivative $J'_{\mathcal{S}(\bg)}$ at $((\bv^\ast, \varphi^\ast), \bg^\ast)$ in the direction of  $(\by_1, y_2)$ is given by
   \begin{align}
       J'_{\mathcal{S}(\bg)}(\ms(\bg^\ast), \bg^\ast)(\by_1, y_2) = &\beta_1\int_{Q}(\bv^\ast - \bv_{Q})\cdot\by_1 dx dt + \beta_2\int_{Q}(\varphi^\ast - \varphi_{Q})y_2 dxdt \no \\ & +\beta_3 \int_{\Omega}(\varphi(T) - \varphi_{\Omega}) y_2(T) dx,\label{1sr part chain}
   \end{align}
   for any $(\by_1, y_2) \in \mathcal{W}$.
   Similarly, we calculate the Gateaux derivative  $J'_{\bg}$ at $((\bv^\ast, \varphi^\ast), \bg^\ast)$ in the direction of $\bu \in \mathcal{U}$ as, 
   \begin{align}
       J'_{\bg}(\mathcal{S}(\bg^\ast), \bg^\ast)(\bu) = \beta_4\int_Q\bg^\ast\cdot\bu \, dxdt,\label{2nd part chain}
   \end{align}
    Using Theorem \ref{thm-f-diff}, we can write  that
    \begin{align}\label{diff  c to s}
    \ms'(\bg^\ast)(\bg-\bg^\ast)=(\bw, \psi).\end{align}
     where $ (\bw, \psi) $ is the unique weak solution of linearized system \eqref{lin-system}, linearized around $ (\bv^\ast, \varphi^\ast,\mu^\ast, w^\ast)$, with control $\bg-\bg^\ast$.
   Now using \eqref{1sr part chain}-\eqref{diff c to s} in \eqref{chain rule} we obtain,
   \begin{align}\label{gd j}
   D\tilde{J}(\bg^\ast)(\bg - \bg^\ast) = &\beta_1\int_{Q}(\bv^\ast - \bv_{Q})\cdot \bw \, dxdt + \beta_2\int_{Q}(\varphi^\ast - \varphi_{Q})\cdot \psi \, dxdt \no \\ &  \quad+ \beta_3\int_{\Omega} (\varphi^\ast(T) - \varphi_{\Omega})\cdot \psi(T) \, dx + \beta_4\int_Q \bg^\ast(\bg - \bg^\ast). 
   \end{align}
 Since $\mathcal{U}_{ad}$ is a nonempty, convex subset of $\mathcal{U}$ and $\tilde{J}$ is G$\hat{\text{a}}$teaux differentiable, then from \cite[Lemma 2.21]{fredi_book}, we have 
 \begin{align}
      \beta_1\int_{Q}&(\bv^\ast - \bv_{Q})\cdot \bw \, dxdt +\beta_2 \int_{Q}(\varphi^\ast - \varphi_{Q})\cdot \psi \, dxdt +\beta_3 \int_{\Omega} (\varphi^\ast(T) - \varphi_{\Omega})\cdot \psi(T) \, dx \no\\&+\beta_4 \int_Q \bg^\ast(\bg - \bg^\ast)\, dx dt  \geq 0 \label{necessary condition}, 
 \end{align}
 for any $\bg\in \U_{ad},$ where $(\bv^\ast, \vphi^\ast, \bg^\ast)$ is as refer in Definition \ref{opt_sol}. 
 The above inequality is very unpleasant as it requires solving linearized system infinitely many times as $\bg$ varies over $\U_{ad}.$ To overcome this situation, we simplify the inequality \eqref{necessary condition} using adjoint variable as follows:  
 \begin{theorem}
     Let $(\bv^\ast, \vphi^\ast, \bg^\ast)$ be the optimal solution as defined in Definition \ref{opt_sol} (note $(\bv^\ast, \vphi^\ast)$ first two component of $\ms(\bg^\ast)$). Then 
     \begin{align}\label{var-equ}
         \int_Q(\beta_4\bg^\ast-\bv^a)\cdot(\bg-\bg^\ast)\geq 0 \quad \text{ for every } \bg\in\U_{ad},
     \end{align}
     where $\bv^a$ is the first component of solution of adjoint system \eqref{adj sys}. Moreover, since $\beta_4>0$ the optimal control $\bg^\ast$ is the $\U$-orthogonal projection of $\frac{1}{\beta_4}\bv^a$ onto $\U_{ad}.$
 \end{theorem}
 \begin{proof}
 We fix $\bg \in \U_{ad}$ and consider the linearized system \eqref{lin-system}, which is linearized around the optimal control $\bg^\ast$ and the corresponding state $(\bv^\ast, \vphi^\ast, \mu^\ast, w^\ast)$, with variation $\bu = \bg - \bg^\ast$. 
We then test these four equations by the adjoint variables $\bv^a, \vphi^a, \mu^a, w^a$, respectively. 
At the same time, we test the adjoint equations with the linearized variables $\bw, \psi, \tht, \xi$, respectively.  
Using integration by parts, we obtain the following identities for a.e.\ $t \in (0, T)$.
   \begin{equation}\label{int by part1}
   \left\{ 
   \begin{aligned}
     &\eta\int_\Omega\nabla\bw :\nabla\bv^a dx+\int_\Omega\lm'(\vphi^\ast)\psi\bv^\ast\cdot\bv^a dx +\int_\Omega\lm(\vphi^\ast)\bw\cdot\bv^a \dx=\int_\Omega\tht\nabla\vphi^\ast\cdot\bv^a dx+\int_\Omega\mu^\ast\nabla\psi\cdot\bv^a dx\\&\qquad+\int_\Omega(\bg-\bg^\ast)\cdot\bv^a dx,\\
     &\int_\Omega\pat\psi\vphi^a dx+\int_\Omega (\bw\cdot\nabla\vphi^\ast)\vphi^a dx+\int_\Omega(\bv^\ast\cdot\nabla\psi)\cdot\vphi^a dx +\int_\Omega\nabla\tht\cdot\nabla\vphi^a dx=-\int_\Omega\sigma\psi\vphi^a dx\\&\qquad+\int_\Omega h'(\vphi^\ast)\psi\vphi^a dx,\\
     &-\int_\Omega\Delta\xi\mu^a dx +\int_\Omega f''(\vphi^\ast)\psi w^\ast\mu^a dx+\int_\Omega f'(\vphi^\ast)\xi\mu^a dx+\nu\int_\Omega\xi\mu^a dx=\int_\Omega\tht\mu^a dx,\\
     &-\int_\Omega\Delta\psi w^a dx+\int_\Omega f'(\vphi^\ast)\psi w^a dx=\int_\Omega\xi w^a dx,
     \end{aligned}\right. \end{equation}for every $(\bv^a, \psi^a, \mu^a, w^a)\in \MV_{\td}\times W\times H\times W$, 
     as well as
     \begin{equation}\label{int by part2}
     \left\{
     \begin{aligned}
        &\eta\int_\Omega\nabla\bv^a:\nabla\bw +\io\lm(\vphi^\ast)\bv^a\cdot\bw dx+\io\vphi^a\nabla\vphi^\ast\cdot\bw dx=-\beta_1\io(\bv-\bv_Q)\cdot\bw dx,\\
       &\io \vphi^a\pat\psi dx+ \io\lm'(\vphi^\ast)\bv^\ast\bv^a\psi dx+\io(\bv^a\cdot\nabla\mu^\ast)\psi dx-\io(\bv^\ast\cdot\nabla\vphi^a)\psi dx+\io\sigma\vphi^a\psi dx\\&-\io h'(\vphi^\ast)\vphi^a\psi dx+\io f''(\vphi^\ast)w^\ast\mu^a\psi dx-\io\Delta\psi w^a dx+\io f'(\vphi^\ast)w^a\psi dx =-\io\beta_2(\vphi-\vphi_Q)\psi dx\\& -\beta_3\io(\vphi^\ast(T)-\vphi_T)\psi(T),\\
       &\io\nabla\tht\cdot\nabla\vphi^a dx-\io(\bv^a\cdot\nabla\vphi^\ast)\tht dx =\io\mu^a\tht dx,\\
        & -\io\Delta\xi\mu^a dx+\io f'(\vphi^\ast)\mu^a\xi dx+\nu\io\mu^a\xi dx= \io w^a\xi dx, 
     \end{aligned}
     \right.\end{equation}
     for every $(\bw, \psi, \tht, \xi)\in \MV_{\td}\times W\times V\times W.$ 
     Note that using Theorem \ref{exist_LM} we can take  $\lm_0\equiv 1$ in \eqref{adj sys}. At this point, we subtract each equation of \eqref{int by part1} from \eqref{int by part2} and integrate over $(0, T)$. Then integrating by parts and after some obvious cancellation and using
  \eqref{necessary condition} yields
  $$\int_Q(\beta_4\bg^\ast-\bv^a)\cdot(\bg-\bg^\ast)\geq 0.$$ The projection formula in the last part of the theorem is a direct consequence of \eqref{var-equ} [see \cite[Theorem 2.33]{fredi_book}].
 \end{proof}
 \subsection{First order necessary optimality condition for $\kappa>0$.} We have prove the existence of an optimal control $\bg^\ast$ in Theorem \ref{exist-op-control} for the problem $\textbf{(CP)}_\eta$. In the sequel, we will assume $\bg^\ast$ is a local solution of $\textbf{(CP)}_\eta$. Now we recall an abstract minimization problem from \cite[Section 3]{CT20}.
 
 Let \(Z\) and \(Y\) be two topological vector spaces and \(\mathcal{K} \subset Z\) and 
\(\mathcal{C} \subset Y\) two convex sets. Given the mappings 
\(G : Z \longrightarrow Y\), \(J : Z \longrightarrow \mathbb{R}\) and 
\(j : Z \longrightarrow (-\infty,+\infty]\), we consider the optimization problem
\[
\text{(Q)} \qquad \min\{ J(z) + j(z) : z \in \mathcal{K} \text{ and } G(z) \in \mathcal{C} \}.
\]

Then the following theorem provides the optimality conditions satisfied by any local solution
of (Q).

\begin{theorem}\cite[Theorem 3.2]{CT20}\label{mcrf}
Let \(\bar{z}\) be a local solution of (Q). Assume that \(J\) and \(G\) are 
Gâteaux differentiable at \(\bar{z}\), \(j\) is convex and continuous at some point of \(\mathcal{K}\), and \(\operatorname{int} \mathcal{C} \neq \emptyset\). Then there exist a real number 
\(\lm_0 \ge 0\), a multiplier \(\wp \in Y'\), and \(\bar{\lambda} \in \partial j(\bar{z})\) 
such that
\begin{align}
&(\lm_0, \wp) \neq (0,0), \\
&\langle \wp,\, y - G(\bar{z}) \rangle_{Y',Y} \le 0 \qquad \forall\, y \in \mathcal{C}, \label{2nd} \\
&\langle \lm_0[J'(\bar{z}) + \bar{\lambda}] + [DG(\bar{z})]^*\wp,\, 
z - \bar{z} \rangle_{Z',Z} \ge 0 \qquad \forall\, z \in \mathcal{K}.\label{3rd} 
\end{align}

Moreover, if the linearized Slater condition
\begin{equation}
\exists z_0 \in \mathcal{K} : G(\bar{z}) + DG(\bar{z})(z_0 - \bar{z}) \in \operatorname{int}\mathcal{C} 
\end{equation}
is satisfied, then \((3.3)\) holds with $\lm_0 = 1$.
\end{theorem}
We will use the above theorem to derive first order necessary optimality condition for $\textbf{(CP)}_\eta$ for the case $\kappa>0.$
\begin{theorem}
    Let $(\bv^\ast, \vphi^\ast, \bg^\ast)$ be a local optimal solution of $\textbf{(CP)}_\eta$ and $j:L^1(Q)\to \rea$ is non-negative, convex and continuous. Then there exist $(\bv^a, \vphi^a, \mu^a, w^a)\in L^2(0, T; \MV_{\td})\times L^2(0, T; W)\times L^2(0, T; H)\times L^2(0, T; W)$ satisfying the adjoint system \eqref{adj sys} and $\varsigma\in\partial j(\bg^\ast)$ such that
     \begin{align}\label{2nd necessary cond}
         \int_Q(\beta_4\bg^\ast+\kappa\varsigma-\bv^a)\cdot(\bg-\bg^\ast)\geq 0 \quad \text{ for every } \bg\in\U_{ad},
     \end{align}
     where $\bv^a$ is the first component of adjoint system \eqref{adj sys}.
\end{theorem}
\begin{proof}
    We apply Theorem \ref{mcrf} to proof this result. In the notation of above theorem we take $\mathcal{K}=\W\times\U_{ad}:=\mathcal{Y},$ $Z=\mathcal{Y},$ $Y=X_2,$ $\mathcal{C}=Y$, and $G$ is defined in \eqref{def-G}. Also, $J$  and $j$ are same as defined in \eqref{cost functional}. As $\mathcal{C}=Y$, the whole space, obviously the linearized slater condition satisfied. 
     Thanks to the discussion on subsection 4.3 and Theorem \ref{exist_LM}, all the hypothesis of Theorem \ref{mcrf} satisfied, there exists $\wp\in X_2'$ and  $\varsigma\in \partial j(\bg^\ast)$ satisfying condition \eqref{2nd} and \eqref{3rd} with $\lm_0=1.$ From \eqref{3rd}, we obtain 
     \begin{align*}
         \langle J'(\vk^\ast, \bg^\ast)+\kappa\varsigma, \bg-\bg^\ast\rangle_{\mathcal{Y}', \mathcal{Y}}+\langle G'((\vk^\ast,\bg^\ast)^t)((\ell, \bg-\bg^\ast)^t), \wp\rangle_{\mathcal{Y}'
         , \mathcal{Y}}\geq 0 \quad \forall (\ell, \bg)\in\mathcal{Y}.
     \end{align*}
  Note that here we have adopted the same notation of $\wp, \, \vk, \, \vk^\ast,$ and $\ell$ from \eqref{short-notation}. Using \eqref{gd j} in the above equality gives
     \begin{align}\label{ast}
         &\beta_1\int_{Q}(\bv^\ast - \bv_{Q})\cdot \bw \, dxdt + \beta_2\int_{Q}(\varphi^\ast - \varphi_{Q})\cdot \psi \, dxdt + \beta_3\int_{\Omega} (\varphi^\ast(T) - \varphi_{\Omega})\cdot \psi(T) \, dx \no \\ &  \quad+ \int_Q (\beta_4\bg^\ast+\kappa\varsigma)(\bg - \bg^\ast)+\langle G'((\vk^\ast,\bg^\ast)^t)((\ell, \bg-\bg^\ast)^t), \wp\rangle_{\mathcal{Y}'
         , \mathcal{Y}}\geq 0 \quad \forall (\ell, \bg)\in\mathcal{Y}.
     \end{align}
     Recalling \eqref{der-G} and noting that $\wp$ satisfies the adjoint equation \eqref{adj sys}, expanding the inner product in \eqref{ast}, using integration by parts and after some cancellation we get
     \begin{align*}
         \int_Q(\beta_4\bg^\ast+\kappa\varsigma-\bv^a)\cdot(\bg-\bg^\ast)\geq 0 \quad \text{ for every } \bg\in\U_{ad},
     \end{align*}
     which conclude the theorem.
    \end{proof}
The convex function $j$ in cost functional accounts for the sparsity of optimal control, that is, the possibility that any locally optimal control may vanish in some region of space time cylinder $Q.$ The form of this region depends on the particular choice of the functional $j$. The sparsity property can be deduce from the variational inequality \eqref{2nd necessary cond} and particular form of subdifferential $\partial j$. For the functional $j$ introduced in \eqref{cost functional} and its subdifferential is given by \cite[Section 0.3.2]{IT79}
\begin{align}\label{subdiff j}
\partial j(\bg) = 
\begin{cases}
\{1\} & \text{if } \bg(x,t) > 0 \, \text{ for a.e. } (x,t) \in Q, \\
\{\varsigma \in L^2(Q) : \varsigma(x,t) \in [-1, 1] \} & \text{ if } \bg(x,t) = 0  \text{ for a.e. } (x,t) \in Q, \\
\{-1\} & \text{if } \bg(x,t) < 0 \, \text{ for a.e. } (x,t) \in Q.
\end{cases}
\end{align}
Using this subdifferential in the variational inequality \eqref{2nd necessary cond}, we obtain the following result:
\begin{theorem}
    Let $\bg^\ast\in\U_{ad}$ be the locally optimal control for the problem $\textbf{(CP)}_\eta$ with $\kappa>0$ and functional $j$ as defined in \eqref{cost functional}. Also, $\ms(\bg^\ast)=(\bv^\ast, \vphi^\ast)$ be the first two component of solution of \eqref{CHB} and the adjoint state $(\bv^a, \vphi^a, \mu^a, w^a)$ solving \eqref{adj sys}. Then there exists a function $\varsigma\in\partial j(\bg^\ast)$ satisfying the variational inequality \eqref{2nd necessary cond}, and it also holds that 
    \begin{align}\label{spasty-equ}
      \|\bv^a\|_{L^2(0, T; \MH)}\leq\kappa, \quad\text{if } \bg^\ast=0.
    \end{align}
    Further, optimal control $\bg^\ast$ can be obtain by the projection formula
    \begin{align}\label{pform_oc}
        \bg^\ast=\mathcal{P}_{\U_{ad}}\left(-\beta_4^{-1}(\kappa\varsigma+\bv^a)\right),
    \end{align}
    where $\mathcal{P}_{\U_{ad}}$ is the orthogonal projection of $\U$ onto $\U_{ad}.$
\end{theorem} 
\begin{proof}
   The projection formula \eqref{pform_oc} is a direct consequence of \eqref{2nd necessary cond} (cf. \cite[Theorem 2.33]{fredi_book}). It remains to show \eqref{spasty-equ}. In \eqref{2nd necessary cond}, we take $\bg=0$ and $\bg=\frac{M\bg^\ast}{\|\bg^\ast\|_{\U}}$, respectively, we deduce that 
   \begin{align}
       \int_Q(\beta_4\bg^\ast+\kappa\varsigma-\bv^a)\cdot\bg^\ast=0.
   \end{align}
  This immediately gives
   \begin{align}
      \|\beta_4\bg^\ast+\bv^a\|_{L^2(Q)}\leq\kappa. 
   \end{align}
   Then, clearly $\|\bv^a\|_{L^2(Q)}\leq \kappa$ if $\bg^\ast=0.$ This completes the proof.
\end{proof}
\begin{remark} In this discussion control is taken in first equation of \eqref{CHB}. The same analysis holds by taking control in second equation of \eqref{CHB}. In this case, the variational inequality \eqref{2nd necessary cond} takes the form 
\begin{align*}
    \int_Q(\beta_4g^\ast+\kappa\varsigma-\vphi^a)(g-g^\ast)\geq 0 \text{ for every }g\in\U_{ad}
\end{align*} where the control and admissible control space taken as $\U=L^2(0, T; H)$, $\U_{ad}=\{g\in\U:\|g\|_{\U}\leq M\}$, respectively, with $M$ being a suitable positive constant. 
\end{remark}


\bibliographystyle{plain} 
\bibliography{mybibliography}
\end{document}